\documentclass[12pt]{amsart}
\usepackage{amsmath,amsthm,amssymb}
\usepackage{enumerate}
\usepackage{mathtools}
\usepackage[all]{xy}

\newtheorem{thm}{Theorem}[section]
\newtheorem*{thm*}{Theorem}
\newtheorem{cor}[thm]{Corollary}
\newtheorem{lem}[thm]{Lemma}

\newtheorem{pro}[thm]{Proposition}

\theoremstyle{definition}
\newtheorem{defin}[thm]{Definition}
\newtheorem{rem}[thm]{Remark}
\newtheorem{exa}[thm]{Example}

\numberwithin{equation}{section}

\newcommand{\N}{\mathbb{N}}

\newcommand{\K}{\mathbb{K}}
\newcommand{\ad}{\mathrm{ad}\,}

\newcommand{\End}{\mathrm{End}}
\newcommand{\id}{\mathrm{id}}

\newcommand{\NA}{\mathfrak{B}}

\newcommand{\re}{\mathrm{re}}

\newcommand{\supp}{\mathrm{supp}\,}
\newcommand{\Sym}{\mathbb{S}}

\newcommand{\ydG}{\prescript{G}{G}{\mathcal{YD}}}

\newcommand{\Y}{\mathcal{Y}}

\newcommand{\charK}{\mathrm{char}\,\K}

\begin{document}


\baselineskip=17pt



\title[Nichols algebras with finite root system of rank two]{
The classification of Nichols algebras over groups with finite root system of rank two}

\author{I. Heckenberger}
\author{L. Vendramin}

\address{I. Heckenberger:
Philipps-Universit\"at Marburg,
FB Mathematik und Informatik,
Hans-Meer\-wein-Stra\ss e,
35032 Marburg, Germany.}
\email{heckenberger@mathematik.uni-marburg.de}

\address{L. Vendramin:
Departamento de Matem\'atica, FCEN, Universidad de Buenos Aires, Pabell\'on 1,
Ciudad Universitaria (1428), Buenos Aires, Argentina.}
\email{lvendramin@dm.uba.ar}

\date{}

\maketitle



\begin{abstract}
    We classify all groups $G$ and all pairs $(V,W)$ of absolutely simple
    Yetter-Drinfeld modules over $G$ such that the support of $V\oplus W$
    generates $G$, $c_{W,V}c_{V,W}\ne\id$, and the Nichols algebra of the
    direct sum of $V$ and $W$ admits a finite root system. As a byproduct, we
    determine the dimensions of such Nichols algebras, and several new families
    of finite-dimensional Nichols algebras are obtained. Our main tool is the
    Weyl groupoid of pairs of absolutely simple Yetter-Drinfeld modules over
    groups.
\end{abstract}

\setcounter{tocdepth}{1}
\tableofcontents{}

\section*{Introduction}

In the last years, Nichols algebras turned out to be important in many branches
of mathematics such as Hopf algebras and quantum groups \cite{MR1637096},
\cite{MR1913436}, \cite{MR2759715}, \cite{MR1632802}, \cite{MR901157,
MR994499}, Schubert calculus \cite{MR1667680}, \cite{MR2209265}, and
mathematical physics \cite{MR2106930}, \cite{MR2965674}. Nichols algebras
appeared first in a work of Nichols \cite{MR0506406}, where he studies and
classifies certain pointed Hopf algebras.  Pointed Hopf algebras have
applications in conformal field theory \cite{MR2030633}. 

Let $\K $ be a field and let $G$ be a group. The Lifting Method of
Andruskiewitsch and Schneider \cite{MR1659895} (see also \cite{MR1913436})
provides the best known approach
to the classification of finite-dimensional pointed Hopf algebras. At the first
place, the method asks to determine all finite-dimensional Nichols algebras
over $G$ and to provide a presentation by generators and relations. Whereas for
abelian groups the situation is understood to a great extent \cite{MR2207786},
\cite{MR2462836}, \cite{MR2630042}, \cite{Ang, Ang_AS},
less is known for non-abelian groups.

One idea to approach the problem is to adapt the method applied for abelian
groups. The problem here is that the structure of the Nichols algebra of a
simple Yetter-Drinfeld module over $G$ is very complicated. Only few
finite-dimensional examples are known, \cite{MR2803792}, \cite{MR2891215}, and
for the important examples of Fomin-Kirillov algebras 
\cite{MR1667680} it is not even known whether they are
Nichols algebras or whether they are finite-dimensional.  Nevertheless, any direct
sum of simple Yetter-Drinfeld modules having a finite-dimensional Nichols algebra
gives rise to the structure of a Weyl
groupoid \cite{MR2766176}, and surprisingly, the finiteness of the Weyl
groupoid implies strong restrictions on the direct summands. Therefore it is
reasonable to attack the classification of semi-simple Yetter-Drinfeld modules
with finite-dimensional Nichols algebras before looking at the simple objects.
The situation is even more astonishing: The functoriality of the Nichols
algebra \cite[Cor. 2.3]{MR1913436} allows to look at Yetter-Drinfeld submodules
of simple objects, which are semi-simple with respect to a smaller group. Then
information in the semi-simple setting can be used for simple objects
\cite{MR2799090}, \cite{MR2786171, MR2745542}.

First ideas to analyze in detail the Nichols algebra of a semi-simple
Yetter-Drinfeld module were developed in \cite{MR2732989}. That work is based
on the notion of the Weyl groupoid of tuples of simple Yetter-Drinfeld modules
over arbitrary Hopf algebras with bijective antipode \cite{MR2766176},
\cite{MR2734956}, \cite{MR3096611}.  Using the classification
of finite Weyl groupoids of rank two \cite{MR2525553}, a breakthrough in the approach
was achieved in \cite{partII}.

Recall that for any group $G$,
a $\K G$-module $V$ is \emph{absolutely simple} if $\mathbb{L}\otimes_\K V$ is a
simple $\mathbb{L}G$-module for any field extension $\mathbb{L}$ of $\K$.
	We say that a Yetter-Drinfeld module $V$ over a group algebra $\mathbb{K}G$
	is absolutely simple, if 
	$\mathbb{L}\otimes_\K V$ is a simple Yetter-Drinfeld module over $\mathbb{L}G$
	for any field extension $\mathbb{L}$ of $\K$.
Recall from \cite{MR2732989} that the groups $\Gamma _n$ are central extensions
of dihedral groups, whereas the group $T$, defined in
	\cite{partII}, is a central extension of
$\mathbf{SL}(2,3)$. 

\begin{thm*}\cite[Thm.\,4.5]{partII}.
	Let $G$ be a non-abelian group, and $V$ and $W$ be two absolutely simple
	Yetter-Drinfeld modules over $G$ such that $G$ is generated by the support of
	$V\oplus W$. Assume that the Nichols algebra of $V\oplus W$ is
  finite-dimensional. If $c_{W,V}c_{V,W}\not=\id _{V\otimes W}$, then $G$ is
	an epimorphic image of $T$ or of $\Gamma_{n}$ for $n\in\{2,3,4\}$.
\end{thm*}

We remark that if the square of the braiding between $V$ and $W$ is
the identity, then $\NA(V\oplus W)$ is isomorphic to $\NA(V)\otimes\NA(W)$
as $\N _0$-graded objects in $\ydG$ by \cite[Thm.  2.2]{MR1779599}.
On the other hand, the assumption that $G$ is
generated by $\supp (V\oplus W)$ is natural since the braiding of $V\oplus W$,
and hence the structure of $\NA (V\oplus W)$ as a braided Hopf algebra,
depends only on the action and coaction of the subgroup of $G$ generated by
$\supp (V\oplus W)$.

Already in \cite{MR2732989}, Nichols algebras of pairs of simple
Yetter-Drinfeld modules over non-abelian epimorphic images of $\Gamma_2$ were
studied and new Nichols algebras of dimension $1296$ over fields of
characteristic $3$ were found. In \cite{examples}, the Nichols algebras over
non-abelian epimorphic images of $T$ and $\Gamma_4$ were studied and new
Nichols algebras of dimensions $80621568$, $262144$ (if $\charK\neq 2$) and
$1259712$, $65536$ (if $\charK=2$) were found.  The situation is more
complicated when $G$ is a non-abelian epimorphic image of $\Gamma_3$, and it is
studied in this work. It is the first case where one meets a
finite-dimensional Nichols algebra not of diagonal type
which has a non-standard Weyl
groupoid. We obtain several new families of Nichols algebras, the ranks and
dimensions of which can be read off from Table~\ref{tab:classification}.

    \begin{table}[h]
        \caption{Nichols algebras with finite root system of rank two}
\begin{center}
        \begin{tabular}{|c|c|c|c|c|c|}
            \hline
            rank & group & dimension & $\charK$ & support & reference\tabularnewline
            \hline 
						$4$ & $\Gamma_{2}$ & $64$ && $Z_2^{2,2}$\rule{0pt}{3ex} & Example~\ref{exa:G2a} \tabularnewline
            \hline
						$4$ & $\Gamma_{2}$ & $1296$ & $3$ & $Z_2^{2,2}$\rule{0pt}{3ex}  & Example~\ref{exa:G2b} \tabularnewline
            \hline
						$4$ & $\Gamma_{3}$ & $10368$ & $\ne2,3$  & $Z_3^{3,1}$ \rule{0pt}{3ex} & Thm.~\ref{thm:P1andP4b}
            \tabularnewline
            \hline
						$4$ & $\Gamma_{3}$ & $5184$ & $2$ & $Z_3^{3,1}$\rule{0pt}{3ex}  & Thm.~\ref{thm:P1andP4b}
            \tabularnewline
            \hline
						$4$ & $\Gamma_{3}$ & $1152$ & $3$ & $Z_3^{3,1}$\rule{0pt}{3ex}  & Thm.~\ref{thm:P1andP4b}
            \tabularnewline
            \hline
						$4$ & $\Gamma_{3}$ & $2239488$ & $2$ & $Z_3^{3,1}$\rule{0pt}{3ex}  & Thms.~\ref{thm:P5},
            \ref{thm:P5''}\tabularnewline
            \hline
						$5$ & $\Gamma_{3}$ & $10368$ & $\ne2,3$ & $Z_3^{3,2}$\rule{0pt}{3ex}  & Thm.~\ref{thm:P1andP4a}
            \tabularnewline
            \hline
						$5$ & $\Gamma_{3}$ & $5184$ & $2$ & $Z_3^{3,2}$\rule{0pt}{3ex}  & Thm.~\ref{thm:P1andP4a}
            \tabularnewline
            \hline
						$5$ & $\Gamma_{3}$ & $1152$ & $3$ & $Z_3^{3,2}$\rule{0pt}{3ex}  & Thm.~\ref{thm:P1andP4a}
            \tabularnewline
            \hline
						$5$ & $\Gamma_{3}$ & $2304$ & & $Z_3^{3,2}$\rule{0pt}{3ex} 
						& Thm.~\ref{thm:P2andP3a}\tabularnewline
						\hline
						$5$ & $\Gamma_{3}$ & $2304$ & & $Z_3^{3,1}$\rule{0pt}{3ex} 
						& Thm.~\ref{thm:P2andP3b}\tabularnewline
            \hline
						$5$ & $\Gamma_{3}$ & $2239488$ & $2$ & $Z_3^{3,2}$\rule{0pt}{3ex}  & Thm.~\ref{thm:P5'}
            \tabularnewline
            \hline
						$5$ & $T$ & $80621568$ & $\ne2$ & $Z_T^{4,1}$\rule{0pt}{3ex}  & Example~\ref{exa:T}\tabularnewline
            \hline
						$5$ & $T$ & $1259712$ & $2$ & $Z_T^{4,1}$\rule{0pt}{3ex}  & Example~\ref{exa:T}\tabularnewline
            \hline
						$6$ & $\Gamma_{4}$ & $262144$ & $\ne2$ & $Z_4^{4,2}$\rule{0pt}{3ex}  &
            Example~\ref{exa:G4}\tabularnewline
            \hline
						$6$ & $\Gamma_{4}$ & $65536$ & $2$ & $Z_4^{4,2}$\rule{0pt}{3ex}  &
            Example~\ref{exa:G4}\tabularnewline
            \hline
        \end{tabular}
        \label{tab:classification}
\end{center}
    \end{table}

Having studied Nichols algebras over non-abelian epimorphic images of
$\Gamma_2$, $\Gamma_3$, $\Gamma_4$, and $T$, we are able to classify all pairs
$(V,W)$ of absolutely simple Yetter-Drinfeld modules over a non-abelian group
$G$ such that the Nichols algebra of $V\oplus W$ is finite-dimensional.
Moreover, we determine
the Hilbert series and the decomposition of the Nichols algebra of $V\oplus W$
into the tensor product of Nichols algebras of simple Yetter-Drinfeld modules.
The finite-dimensional Nichols algebras appearing in our
classification are listed in Table \ref{tab:classification}. The 
pairs $(V,W)$ of absolutely simple Yetter-Drinfeld modules over $G$ appear in Section \ref{section:examples}.
Our main theorem is the following.

\begin{thm*}
    Let $G$ be a non-abelian group and $V$ and $W$ be two absolutely simple Yetter-Drinfeld
    modules over $G$ such that $G$ is generated by the support of $V\oplus W$. 
    Assume that 
    $(\id-c_{W,V}c_{V,W})(V\otimes W)$ is non-zero 
    and that the Nichols algebra $\NA(V\oplus W)$ is
    finite-dimensional. Then $\NA(V\oplus W)$ is one of the
    Nichols algebras of Table \ref{tab:classification}.
\end{thm*}


See Theorem~\ref{thm:big} for the more precise statement.
Let us explain briefly how the proof of Theorem~\ref{thm:big} goes.  We have
to study in detail Nichols algebras over non-abelian epimorphic images of the
groups $\Gamma_2$, $\Gamma _3$, $\Gamma_4$ and $T$. The analysis concerning the
group $\Gamma_2$ was done in \cite{MR2732989} and the groups $\Gamma_4$ and $T$
were studied in \cite{examples}. The classification of finite-dimensional
Nichols algebras associated with $\Gamma_3$ is one of the main results of this
paper and requires several steps. We need to deal with three different pairs
$(V,W)$ of absolutely simple Yetter-Drinfeld modules over non-abelian
epimorphic images of $\Gamma_3$.  We first determine when $(\ad V)^m(W)$ and
$(\ad W)^m(V)$ are absolutely simple or zero and then we compute the Cartan
matrix of $(V,W)$. Then we prove that these pairs are essentially the only
pairs which we need to consider, and the reflections of these pairs are
computed.  With this information we compute the finite root systems of rank two
associated with Nichols algebras over non-abelian epimorphic images of
$\Gamma_3$. This information allows us to determine the structure of such
Nichols algebras.

The main result of our paper is expected to lead to powerful applications.  We
intend to attack the classification of finite-dimensional Nichols algebras of
finite direct sums of absolutely simple Yetter-Drinfeld modules over groups.
For this project it is very useful that the reflections of the absolutely
simple pairs are already calculated.  On the other hand, we are confident that
our classification will be useful to study Nichols algebras over simple
Yetter-Drinfeld modules, as it was done for example in \cite{MR2786171,
MR2745542}.

We do not know the defining relations of the Nichols algebras appearing in our
classification.  In the spirit of \cite[Question 5.9]{MR1907185}, one then has the
following problem: Give a nice presentation by generators and relations of
the Nichols algebras appearing in Table~\ref{tab:classification}. 
To attack this, the ideas of \cite{Ang, Ang_AS} could be useful.

\medskip
The paper is organized as follows.  In Section \ref{section:examples} we list
all the finite-dimensional Nichols algebras appearing in our classification.
In Section \ref{section:main} we state the main result of the paper, Theorem
\ref{thm:big}. This theorem classifies Nichols algebras of group type over the
sum of two absolutely simple Yetter-Drinfeld modules. Sections
\ref{section:preliminaries}, \ref{section:1}, \ref{section:2}, \ref{section:3},
\ref{section:reflections} and \ref{section:NicholsG3} are devoted to understand
the structure and the root systems of finite-dimensional Nichols algebras over
non-abelian epimorphic images of $\Gamma_3$.  Finally, in Section
\ref{section:bigproof}, we prove our 
main result, Theorem \ref{thm:big}.

\section{The examples}
\label{section:examples}

Before stating our main result, we collect all the examples of Nichols algebras
with finite root systems obtained over non-abelian epimorphic images of $\Gamma_2$,
$\Gamma_3$, $\Gamma_4$ and $T$. These are the examples which appear in our
classification in Theorem~\ref{thm:big}.

Recall from~\cite{MR2732989} that the group $\Gamma _n$ for $n\ge 2$ is the
group given by generators $a,b,\nu $ and relations
\[ ba=\nu ab,\quad \nu a=a\nu^{-1},\quad \nu b=b\nu,\quad \nu ^n=1, \]
and $T$ is the group
given by generators $\zeta ,\chi _1,\chi _2$ and relations
\[
  \zeta \chi_1=\chi_1\zeta ,\quad
  \zeta \chi_2=\chi_2\zeta ,\quad
  \chi_1\chi_2\chi_1=\chi_2\chi_1\chi_2,\quad
  \chi_1^3=\chi_2^3.
\]
\begin{rem}
The groups $\Gamma_2$, $\Gamma_3$, $\Gamma_4$ and $T$ are
isomorphic to the enveloping groups of the quandles $Z_2^{2,2}$, $Z_3^{3,1}$,
$Z_3^{3,2}$, $Z_4^{4,2}$ and $Z_T^{4,1}$, see \cite[\S2]{partII} and
\cite{examples} for an alternative description of these quandles. An
epimorphic image $G$ of any of these enveloping groups $G_X$ is non-abelian if and
only if the restriction of the epimorphism $G_X\to G$ to the quandle $X$ is injective.
\end{rem}

By \cite{MR2732989}, the group $\Gamma_3$ is
isomorphic to the group given by generators $\nu$, $\zeta$ and $\gamma$
and relations
\begin{align*}
  \gamma\nu=\nu^2\gamma, \quad \zeta\gamma=\gamma\zeta, \quad
  \zeta\nu=\nu\zeta, \quad \nu^3=1.
\end{align*}

\subsection{Epimorphic images of $\Gamma_2$}
\label{subsection:G2}

In \cite[\S4]{MR2732989}, Nichols algebras over non-abelian epimorphic images of
$\Gamma_2$ were studied. Let $G$
be a non-abelian group. Let
$g,h,\epsilon\in G$, and assume that there is a group epimorphism
\[ \Gamma_2\to G ,\qquad a\mapsto g,\quad b\mapsto h,\quad \nu \mapsto
\epsilon .
\]

\begin{exa}
	\label{exa:G2a}
	Let $V,W\in \ydG $. Assume that $V\simeq M(g,\rho)$,
  where $\rho$ is a character of $G^g=\langle \epsilon,g,h^2\rangle$,
  and $W\simeq M(h,\sigma)$,
  where $\sigma$ is a character of $G^h=\langle \epsilon,h,g^2\rangle$.  
	Let $v\in V_g$ with $v\ne0$. Then $\{v,hv\}$ is a basis of $V$ and the degrees
	of these basis vectors are $g$ and $\epsilon g$, respectively.  Similarly, let
	$w\in W_h$ with $w\ne0$. Then $\{w,gw\}$ is a basis of $W$ and the degrees of
	these vectors are $h$ and $\epsilon h$, respectively.  
	The action of $G$ on $V$ and $W$ is given by the following tables:
	\begin{center}
		\begin{tabular}{c|cc}
			$V$ & $v$ & $hv$ \tabularnewline
			\hline
			$\epsilon$ & $\rho(\epsilon)v$ & $\rho(\epsilon)hv$ \tabularnewline
			$h$ & $hv$ & $\rho(h^2)v$ \tabularnewline
			$g$ & $\rho(g)v$ & $\rho(\epsilon)\rho(g)hv$ \tabularnewline
		\end{tabular}
		\qquad
		\begin{tabular}{c|cc}
			$W$ & $w$ & $gw$ \tabularnewline
			\hline
			$\epsilon$ & $\sigma(\epsilon)w$ & $\sigma(\epsilon)gw$ \tabularnewline
			$h$ & $\sigma(h)w$ & $\sigma(\epsilon)\sigma(h)gw$ \tabularnewline
			$g$ & $gw$ & $\sigma(g^2)w$ \tabularnewline
		\end{tabular}
	\end{center}
	Assume that 
	\[
		\rho(\epsilon h^2)\sigma(\epsilon g^2)=1,\quad\rho(g)=\sigma(h)=-1.
	\]
    Then, by \cite[Thm.\,4.6]{MR2732989},
    $\dim\NA(V\oplus W)=64$ and the Hilbert series of the Nichols algebra
    $\NA(V\oplus W)$ is
    \[
    \mathcal{H}(t_1,t_2)=(1+t_1)^2(1+t_1t_2)^2(1+t_2)^2.
    \]
    A special case of this example appeared first in \cite[Example 6.5]{MR1800714}.
\end{exa}

\begin{exa}
	\label{exa:G2b}
	Let $V,W\in \ydG $. Assume that $V\simeq M(g,\rho)$,
  where $\rho$ is a character of $G^g=\langle \epsilon,g,h^2\rangle$,
  and $W\simeq M(h,\sigma)$,
  where $\sigma$ is a character of $G^h=\langle \epsilon,h,g^2\rangle$.
  Assume also that $\charK=3$ and that
	\begin{align}
		\rho(\epsilon h^2)\sigma(\epsilon g^2)=1,\quad \rho(g)=1,\quad \sigma(h)=-1.
	\end{align}
  Then, by \cite[Thm.\,4.7]{MR2732989},
	$\dim\NA(V\oplus W)=1296$ and
	the Hilbert series of the Nichols algebra $\NA(V\oplus W)$ is
	\[
		\mathcal{H}(t_1,t_2)=(1+t_1+t_1^2)^2(1+t_2)^2
    (1+t_1t_2+t_1^2t_2^2)^2(1+t_1^2t_2)^2.
	\]
\end{exa}

\begin{rem}
    The braiding of the Yetter-Drinfeld modules of Examples~\ref{exa:G2a} and
    \ref{exa:G2b} can be obtained from the following table:
    \begin{center}
        \begin{tabular}{c|cccc}
            & $v$ & $hv$ & $w$ & $gw$\tabularnewline
            \hline 
            $g$ & $\rho(g)v$ & $\rho(\epsilon g)hv$ & $gw$ & $\sigma(g^{2})w$\tabularnewline
            $\epsilon g$ & $\rho(\epsilon g)v$ & $\rho(g)hv$ &
						$\sigma(\epsilon)gw$ & $\sigma(\epsilon g^{2})w$\tabularnewline
            $h$ & $hv$ & $\rho(h^{2})v$ & $\sigma(h)w$ & $\sigma(\epsilon h)gw$\tabularnewline
            $\epsilon h$ & $\rho(\epsilon)hv$ & $\rho(\epsilon h^{2})v$ &
						$\sigma(\epsilon h)w$ & $\sigma(h)gw$\tabularnewline
        \end{tabular}
    \end{center}
\end{rem}

\subsection{Epimorphic images of $\Gamma_4$}
\label{subsection:G4}

Finite-dimensional Nichols algebras over non-abelian epimorphic images of
$\Gamma_4$ were computed in \cite[\S5]{examples}.  Let $G$ be a non-abelian
group and let $g,h,\epsilon\in G$. Assume that there is a group epimorphism
\[ \Gamma_4\to G ,\qquad a\mapsto g,\quad b\mapsto h,\quad \nu \mapsto
\epsilon 
\]
such that $\epsilon^2\ne1$.

\begin{exa} \label{exa:G4}
	Let $V,W\in \ydG $. Assume that $V\simeq M(h,\rho)$,
  where $\rho$ is a character of the
	centralizer $G^h=\langle \epsilon,h,g^2\rangle$ with $\rho(h)=-1$.
  Let $v\in V_h$ with $v\ne0$. Then the elements $v$, $gv$ form a basis of
  $V$, and the degrees of
	these basis vectors are $h$ and $ghg^{-1}=\epsilon^{-1}h$, respectively.  The
	action of $G$ on $V$ is given by the following table: 

	\begin{center}
		\begin{tabular}{c|cc}
			$V$ & $v$ & $gv$ \tabularnewline
			\hline
			$\epsilon$ & $\rho(\epsilon)v$ & $\rho(\epsilon)^{-1}gv$ \tabularnewline
			$h$ & $-v$ & $-\rho(\epsilon)^{-1}gv$ \tabularnewline
			$g$ & $gv$ & $\rho(g^2)v$ \tabularnewline
		\end{tabular}
	\end{center}

	Assume that $W\simeq M(g,\sigma)$, where $\sigma$ is a
	character of $G^g=\langle \epsilon^2,\epsilon^{-1}h^2,g\rangle$ with
	$\sigma(g)=-1$. Let $w\in W_g$ with $w\ne0$. The elements $w$, $hw$, $\epsilon
	w$, $\epsilon hw$ form a basis of $W$. The degrees of these basis vectors are
	$g$, $\epsilon g$, $\epsilon^2g$ and $\epsilon^3 g$, respectively.  The action
	of $G$ on $W$ is given by the following table:

	\begin{center}
		\begin{tabular}{c|cccc}
			$W$ & $w$ & $hw$ & $\epsilon w$ & $\epsilon hw$ \tabularnewline
			\hline
			$\epsilon$ & $\epsilon w$ & $\epsilon hw$ & $\sigma(\epsilon^2)w$ &
			$\sigma(\epsilon^2)hw$ \tabularnewline
			$h$ & $hw$ & $\sigma(\epsilon^{-1}h^2)\epsilon w$ & $\epsilon hw$ &
			$\sigma(\epsilon h^2)w$ \tabularnewline
			$g$ & $-w$ & $-\sigma(\epsilon^2)\epsilon hw$ &
			$-\sigma(\epsilon^2)\epsilon w$ & $-\sigma(\epsilon^2)hw$
			\tabularnewline
		\end{tabular}
	\end{center}
	Assume further that 
	\[
	\rho(\epsilon)=\rho(g^2)\sigma(\epsilon^{-1}h^2),
	\quad
	\rho(\epsilon)^2=-1.
	\]
  Then, by \cite[Thm.\,5.4]{examples}, 
	\begin{align*}
		\mathcal{H}(t_1,t_2)=
		(1+t_2)^4(1+t_2^2)^2(1+t_1t_2)^4(1+t_1^2t_2^2)^2q(t_1t_2^2)q(t_1),
	\end{align*}
	where
	\begin{gather*}
		q(t)=\begin{cases}
			(1+t)^2(1+t^2) & \text{if $\charK\ne2$},\\
			(1+t)^2 & \text{if $\charK=2$}.
		\end{cases}
	\end{gather*}
	In particular, 
	\[
	\dim \NA (V\oplus W)=\begin{cases}
		8^264^2=262144 & \text{if $\charK\ne2$},\\
		4^264^2=65536 & \text{if $\charK=2$}.
	\end{cases}
	\]
\end{exa}

\begin{rem}
    The braiding of the Yetter-Drinfeld module of Example~\ref{exa:G4} can be
		obtained from the following table:
    \begin{center}
        \begin{tabular}{c|cccccc}
            & $v$ & $gv$ & $w$ & $hw$ & $\epsilon w$ & $\epsilon hw$\tabularnewline
            \hline 
            $h$ & $-v$ & $-\rho(\epsilon)^{-1}gv$ & $hw$ & $\sigma(\epsilon^{-1}h^{2})\epsilon w$ & $\epsilon hw$ & $\sigma(\epsilon h^{2})w$\tabularnewline
            $\epsilon^{3}h$ & $-\rho(\epsilon)^{3}v$ & $-gv$ & $\sigma (\epsilon
						^2)\epsilon hw$ & $\sigma (\epsilon ^{-1}h^2)w$ & $hw$ &
						$\sigma(\epsilon ^{-1}h^2)\epsilon w$\tabularnewline
            $g$ & $gv$ & $\rho(g^{2})v$ & $-w$ & $-\sigma(\epsilon^{2})\epsilon hw$ & $-\sigma(\epsilon^{2})\epsilon w$ & $-\sigma(\epsilon^{2})hw$\tabularnewline
            $\epsilon g$ & $\rho(\epsilon^3)gv$ & $\rho(\epsilon g^2)v$ & $-\epsilon w$ & $-hw$ & $-w$ & $-\sigma(\epsilon^{2})\epsilon hw$\tabularnewline
            $\epsilon^{2}g$ & $\rho(\epsilon^2)gv$ & $\rho(\epsilon^{2}g^2)v$ & $-\sigma(\epsilon^{2})w$ & $-\epsilon hw$ & $-\epsilon w$ & $-hw$\tabularnewline
            $\epsilon^{3}g$ & $\rho(\epsilon)gv$ & $\rho(\epsilon^{3}g^2)v$ & $-\sigma(\epsilon^{2})\epsilon w$ & $-\sigma(\epsilon^{2})hw$ & $-\sigma(\epsilon^{2})w$ & $-\epsilon hw$\tabularnewline
        \end{tabular}
    \end{center}
\end{rem}

\subsection{Epimorphic images of $T$}
\label{subsection:T}

Nichols algebras over non-abelian epimorphic images of the group $T$ were
studied in \cite[\S2]{examples}. Let $G$ be a non-abelian group, and let
$x_1,x_2,z\in G$. Assume that there is a group epimorphism
\[
  T\mapsto G,\qquad \zeta \mapsto z, \quad \chi_1\mapsto x_1,\quad
  \chi_2\mapsto x_2.
\]
Clearly, $z$ is a central element of $G$. Moreover, the elements 
\[
	x_1,\;
	x_2,\;
	x_3\coloneqq x_2x_1x_2^{-1},\;
	x_4\coloneqq x_1x_2x_1^{-1}
\]
form a conjugacy class of $G$. 

\begin{exa}
	\label{exa:T}
	Let $V,W\in \ydG $. Assume that $V\simeq M(z,\rho)$,
  where $\rho$ is a character of the centralizer $G^z=G$, and
  $W=M(x_1,\sigma)$, where $\sigma$ is a character of
  $G^{x_1}=\langle x_1,x_2x_3,z\rangle$ with
  $\sigma(x_1)=-1$ and $\sigma(x_2x_3)=1$.  Let $v\in V_z\setminus \{0\}$. Then
	$\{v\}$ is basis of $V$.  The action of $G$ on $V$ is given by
	\[
	zv=\rho(z)v,\quad 
	x_iv=\rho (x_1)v\quad\text{for all $i\in\{1,2,3,4\}$.}
	\]
	Let $w_1\in W_{x_1}$ such that $w_1\ne0$. Then the vectors 
	\[
	w_1,\;
	w_2\coloneqq -x_4w_1,\;
	w_3\coloneqq -x_2w_1,\;
	w_4\coloneqq -x_3w_1
	\]
	form a basis of $W$. The degrees of these vectors are $x_1$, $x_2$, $x_3$ and
	$x_4$, respectively. The action of $G$ on $W$ is given by the following table:
	\begin{center}
		\begin{tabular}{c|cccc}
			$W$ & $w_{1}$ & $w_{2}$ & $w_{3}$ & $w_{4}$\tabularnewline
			\hline
			$x_{1}$ & $-w_{1}$ & $-w_{4}$ & $-w_{2}$ & $-w_{3}$\tabularnewline
			$x_{2}$ & $-w_{3}$ & $-w_{2}$ & $-w_{4}$ & $-w_{1}$\tabularnewline
			$x_{3}$ & $-w_{4}$ & $-w_{1}$ & $-w_{3}$ & $-w_{2}$\tabularnewline
			$x_{4}$ & $-w_{2}$ & $-w_{3}$ & $-w_{1}$ & $-w_{4}$\tabularnewline
			$z$ & $\sigma(z) w_1$ & $\sigma(z) w_2$ & $\sigma(z) w_3$ & $\sigma(z) w_4$\tabularnewline 
		\end{tabular}
	\end{center}
	Assume further that 
	\begin{align*}
		(\rho (x_1)\sigma (z))^2-\rho (x_1)\sigma (z)+1=0, \quad
		\rho (x_1z)\sigma (z)=1.
	\end{align*}
  Then, by \cite[Thm.\,2.8]{examples},
  $\NA(V\oplus W)$ is finite-dimensional. If $\charK\ne 2$, then
  \[
    \mathcal{H}(t_1,t_2)=(6)_{t_1}(6)_{t_1t_2^3}(6)_{t_1^2t_2^3}(2)_{t_2}^2(3)_{t_2}(6)_{t_2}
    (2)_{t_1t_2}^2 
    (3)_{t_1t_2}(6)_{t_1t_2}(2)_{t_1t_2^2}^2(3)_{t_1t_2^2}(6)_{t_1t_2^2}
  \]
	and $\dim \NA (V\oplus W)=6^3\,72^3=80621568$,
	and if $\charK=2$, then
	\begin{align*}
		\mathcal{H}(t_1,t_2)=(3)_{t_1}(3)_{t_1t_2^3}
    (3)_{t_1^2t_2^3}(2)_{t_2}^2(3)_{t_2}^2
		(2)_{t_1t_2}^2(3)_{t_1t_2}^2 (2)_{t_1t_2^2}^2(3)_{t_1t_2^2}^2
	\end{align*}
	and $\dim \NA (V\oplus W)=3^3\,36^3=1259712$.
\end{exa}

\begin{rem}
    The structure of the Yetter-Drinfeld module of Example~\ref{exa:T} can be
    read off the following table:
    \begin{center}
        \begin{tabular}{c|cccccc}
            & $v$ & $w_{1}$ & $w_{2}$ & $w_{3}$ & $w_{4}$ & \tabularnewline
            \hline 
            $z$ & $\rho(z)v$ & $\sigma(z)w_{1}$ & $\sigma(z)w_{2}$ & $\sigma(z)w_{3}$ & $\sigma(z)w_{4}$ & \tabularnewline
            $x_{1}$ & $\rho(x_{1})v$ & $-w_{1}$ & $-w_{4}$ & $-w_{2}$ & $-w_{3}$ & \tabularnewline
            $x_{2}$ & $\rho(x_{1})v$ & $-w_{3}$ & $-w_{2}$ & $-w_{4}$ & $-w_{1}$ & \tabularnewline
            $x_{3}$ & $\rho(x_{1})v$ & $-w_{4}$ & $-w_{1}$ & $-w_{3}$ & $-w_{2}$ & \tabularnewline
            $x_{4}$ & $\rho(x_{1})v$ & $-w_{2}$ & $-w_{3}$ & $-w_{1}$ & $-w_{4}$ & \tabularnewline
        \end{tabular}
    \end{center}
\end{rem}

\subsection{Epimorphic images of $\Gamma_3$}
\label{subsection:G3}

The results of this section will be proved in Section \ref{section:NicholsG3}.
Let $G$ be a non-abelian group. Let $g,\epsilon ,z\in G$, and assume that there
is a group epimorphism
\[
  \Gamma_3\to G,\qquad
  \gamma \mapsto g,\quad \nu \mapsto \epsilon,\quad \zeta \mapsto z.
\]

\begin{exa}
    \label{exa:Z32}
    Let $V\in \ydG $. Assume that $V\simeq M(g,\rho)$,
    where $\rho$ is a character of
    $G^g=\langle g,z\rangle$. Let $v\in V_g$ with $v\ne0$.  The elements $v$,
    $\epsilon v$ and $\epsilon^2v$ form a basis of $V$. The degrees of these
    vectors are $g$, $g\epsilon$ and $g\epsilon^2$, respectively.  

    Similarly, let $W\in \ydG $ such that $W\simeq M(\epsilon z,\sigma)$, where
    $\sigma$ is a character of the centralizer $G^\epsilon=\langle\epsilon, z,
    g^2\rangle$.  Let $w\in W_{\epsilon z}$ with $w\ne0$.  Then $w$, $gw$ is a
    basis of $W$.  The degrees of these basis vectors are $\epsilon z$ and
    $\epsilon^2z$, respectively. The actions of $G$ on $V$ and $W$ are given in
    the following tables:
    \begin{center}
        \begin{tabular}{c|ccc}
            $V$ & $v$ & $\epsilon v$ & $\epsilon^2 v$ \tabularnewline
            \hline
            $\epsilon$ & $\epsilon v$ & $\epsilon^2 v$ & $v$\tabularnewline
            $z$ & $\rho(z)v$ & $\rho(z)\epsilon v$ & $\rho(z)\epsilon^2v$\tabularnewline
            $g$ & $\rho(g)v$ & $\rho(g)\epsilon^2v$ & $\rho(g)\epsilon v$\tabularnewline
        \end{tabular}
        \qquad
        \begin{tabular}{c|cc}
            $W$ & $w$ & $gw$\tabularnewline
            \hline
            $\epsilon$ & $\sigma(\epsilon)w$ & $\sigma(\epsilon)^2gw$ \tabularnewline
            $z$ &  $\sigma(z)w$ & $\sigma(z)gw$ \tabularnewline
            $g$ &  $gw$ & $\sigma(g^2)w$ \tabularnewline
        \end{tabular}
    \end{center}

    If $\rho (g)=\sigma (\epsilon z)=-1$, $\rho (z)^2\sigma (\epsilon
    g^2)=1$, and $1+\sigma(\epsilon)+\sigma(\epsilon)^2=0$, then 
    \[
    \dim\NA(V\oplus W)=\begin{cases}
        10368 & \text{if $\charK\not\in\{2,3\}$},\\
        5184 & \text{if $\charK=2$},\\
        1152 & \text{if $\charK=3$},
    \end{cases}
    \]
    and 
    \[
		 \mathcal{H}(t_1,t_2)=(2)_{t_2}(h'_p)_{t_2}(2)^2_{t_1t_2}
     (3)_{t_1t_2}(h_p)_{t_1^2t_2}(2)_{t_1}^2(3)_{t_1},
    \]
    where $p=\charK $,
    $h_2=3$, $h_3=2$, and $h_p=6$ for all $p\not\in\{2,3\}$, and $h'_3=2$,
    $h'_p=6$ for all $p\ne3$, by Theorem~\ref{thm:P1andP4a}.
    
    If $\rho (g)=\sigma (\epsilon z)=-1$, $\rho (z)^2\sigma (\epsilon
    g^2)=\sigma(\epsilon)=1$, and $\charK \ne 3$, then 
    \[
    \mathcal{H}(t_1,t_2)=(2)^2_{t_2}(2)^2_{t_1t_2}(3)_{t_1t_2}(2)^2_{t_1^2t_2}(2)^2_{t_1}(3)_{t_1} 
    \]
    and $\dim\NA(V\oplus W)=2304$,
    by Theorem~\ref{thm:P2andP3a}.

    If $\charK=2$, $\rho (g)=1$, $(3)_{\sigma (\epsilon )}=0$,
    $\sigma (z)=\sigma (\epsilon )$, and $\rho (z)^2\sigma (\epsilon g^2)=1$,
    then 
    $\dim\NA(V\oplus W)=2239488$ and 
    \[
    \mathcal{H}(t_1,t_2)=
    (3)_{t_2}^2 
    (2)_{t_1t_2^2}^2(3)_{t_1t_2^2}
    (2)_{t_1^2t_2^3}
    (3)_{t_1t_2}(4)_{t_1t_2}(6)_{t_1t_2}(6)_{t_1^2t_2^2}
    (2)_{t_1^2t_2}
    (2)_{t_1}^2(3)_{t_1}, 
    \]
    by Theorem~\ref{thm:P5'}.
\end{exa}

\begin{exa} \label{exa:Z31a}
    Let $V\in \ydG $. Assume that $V\simeq M(g,\rho)$,
    where $\rho$ is a character of
    $G^g=\langle g,z\rangle$. Let $v\in V_g$ with $v\ne0$.  The elements $v$,
    $\epsilon v$ and $\epsilon^2v$ form a basis of $V$. The degrees of these
    vectors are $g$, $g\epsilon$ and $g\epsilon^2$, respectively.  

    Let $W\in \ydG $ such that $W\simeq M(z,\sigma)$,
    where $\sigma$ is a character of $G$.
    Let $w\in W_z$ with $w\ne0$.  Then $w$ is a basis of $W$.
    The action of $G$ on $V$ can be obtained from Example~\ref{exa:Z32}.

    Let $p=\charK $. If $\rho (g)=-1$, $(3)_{-\rho (z)\sigma (g)}=0$, and
    $\rho (z)\sigma (gz)=1$, then
    \[
    \dim\NA(V\oplus W)=\begin{cases}
        10368 & \text{if $\charK\not\in\{2,3\}$},\\
        5184 & \text{if $\charK=2$},\\
        1152 & \text{if $\charK=3$},
    \end{cases}
    \]
    and 
    \[
    \mathcal{H}(t_1,t_2)=(h_p)_{t_2}(2)_{t_1t_2}^2(3)_{t_1t_2}
    (2)_{t_1^2t_2}(h'_p)_{t_1^2t_2}(2)_{t_1}^2(3)_{t_1}
    \]
    by Theorem~\ref{thm:P1andP4b}, where $h_p$ and $h'_p$ are as in Example~\ref{exa:Z32}.

    If $\charK=2$, $\sigma (z)=1$, $(3)_{\rho (z)\sigma (g)}=0$, and $(\rho
    (g)-1)(\rho (gz)\sigma (g)-1)=0$, then $\dim\NA(V\oplus W)=2239488$, and 
    \[
    \mathcal{H}(t_1,t_2)=(2)_{t_2}(3)_{t_1t_2}(4)_{t_1t_2}(6)_{t_1t_2}
    (6)_{t_1^2t_2^2}(2)_{t_1^4t_2^3}
    (2)_{t_1^3t_2^2}^2(3)_{t_1^3t_2^2}(3)_{t_1^2t_2}^2
    (2)_{t_1}^2(3)_{t_1},
    \]
    or 
    \[
    \mathcal{H}(t_1,t_2)=
    (2)_{t_2}
    (2)_{t_1t_2}^2(3)_{t_1t_2}
    (3)_{t_1^2t_2}^2
    (2)_{t_1^3t_2}^2(3)_{t_1^3t_2}
    (2)_{t_1^4t_2}
    (3)_{t_1}(4)_{t_1}(6)_{t_1}(6)_{t_1^2}, 
    \]
    by Theorems~\ref{thm:P5} and \ref{thm:P5''}.
\end{exa}

\begin{exa} \label{exa:Z31b}
    Let $V\in \ydG $. Assume that $V\simeq M(g,\rho)$,
    where $\rho$ is a character of
    $G^g=\langle g,z\rangle$. Let $v\in V_g$ with $v\ne0$.  The elements $v$,
    $\epsilon v$ and $\epsilon^2v$ form a basis of $V$. The degrees of these
    vectors are $g$, $g\epsilon$ and $g\epsilon^2$, respectively.  
    The action of $G$ on $V$ can be obtained from Example~\ref{exa:Z32}.

    Let $W\in \ydG $ such that $W\simeq M(z,\sigma)$,
    where $\sigma$ is an absolutely irreducible representation of $G$
    of degree $\ge 2$.
    Then $\charK \not=3$, $\deg \sigma =2$,
    $\sigma (1+\epsilon +\epsilon ^2)=0$, and the
    isomorphism class of $W$ is uniquely determined by the constants
    $\sigma (g^2)$ and $\sigma (z)$, see
    Lemma~\ref{lem:simpleG3modules}.

    Assume that $\rho (g)=\sigma (z)=-1$ and $\rho (z^2)\sigma (g^2)=1$.  Then 
    \[
       \mathcal{H}(t_1,t_2)=(2)_{t_2}^2(2)_{t_1t_2}^2(3)_{t_1t_2}(2)^2_{t_1^2t_2}(2)^2_{t_1}(3)_{t_1},
    \]
    and $\dim\NA(V\oplus W)=2304$ by Theorem~\ref{thm:P2andP3b}.
\end{exa}

\begin{rem}
    The braidings of the Yetter-Drinfeld modules of Examples~\ref{exa:Z32},
    \ref{exa:Z31a}, and
	  \ref{exa:Z31b}, respectively,
		can be obtained from the following tables:
    \begin{center}
        \begin{tabular}{c|ccccc}
            & $v$ & $\epsilon v$ & $\epsilon^{2}v$ & $w$ & $gw$\tabularnewline
            \hline 
            $g$ & $\rho(g)v$ & $\rho(g)\epsilon^{2}v$ & $\rho(g)\epsilon v$ & $gw$ & $\sigma(g^{2})w$\tabularnewline
            $g\epsilon$ & $\rho(g)\epsilon^{2}v$ & $\rho(g)\epsilon v$ & $\rho(g)v$ & $\sigma(\epsilon)gw$ & $\sigma(\epsilon ^{2} g^{2})w$\tabularnewline
            $g\epsilon^{2}$ & $\rho(g)\epsilon v$ & $\rho(g)v$ &
						$\rho(g)\epsilon^{2}v$ & $\sigma(\epsilon)^{2}gw$ & $\sigma(
						\epsilon g^{2})w$\tabularnewline
            $\epsilon z$ & $\rho(z)\epsilon v$ & $\rho(z)\epsilon^{2}v$ &
						$\rho(z)v$ & $\sigma(\epsilon z)w$ & $\sigma(\epsilon^{2}z)gw$\tabularnewline
            $\epsilon^{2}z$ & $\rho(z)\epsilon^{2}v$ & $\rho(z)v$ &
						$\rho(z)\epsilon v$ & $\sigma(\epsilon^{2}z)w$ & $\sigma(\epsilon z)gw$\tabularnewline
        \end{tabular}
    \end{center}
    \begin{center}
        \begin{tabular}{c|cccc}
            & $v$ & $\epsilon v$ & $\epsilon^{2}v$ & $w$\tabularnewline
            \hline 
            $g$ & $\rho(g)v$ & $\rho(g)\epsilon^{2}v$ & $\rho(g)\epsilon v$ & $\sigma(g)w$\tabularnewline
            $g\epsilon$ & $\rho(g)\epsilon^{2}v$ & $\rho(g)\epsilon v$ & $\rho(g)v$ & $\sigma(g\epsilon)w$\tabularnewline
            $g\epsilon^{2}$ & $\rho(g)\epsilon v$ & $\rho(g)v$ & $\rho(g)\epsilon^{2}v$ & $\sigma(g\epsilon^{2})w$\tabularnewline
            $z$ & $\rho(z)v$ & $\rho(z)\epsilon v$ & $\rho(z)\epsilon^{2}v$ & $\sigma(z)w$\tabularnewline
        \end{tabular}
    \end{center}
    \begin{center}
        \begin{tabular}{c|ccccc}
            & $v$ & $\epsilon v$ & $\epsilon^{2}v$ & $w$ & $gw$\tabularnewline
            \hline 
            $g$ & $\rho(g)v$ & $\rho(g)\epsilon^{2}v$ & $\rho(g)\epsilon v$ & $gw$ & $\sigma(g^{2})w$\tabularnewline
            $g\epsilon$ & $\rho(g)\epsilon^{2}v$ & $\rho(g)\epsilon v$ & $\rho(g)v$ & $\lambda gw$ & $\lambda^2\sigma(g^{2})w$\tabularnewline
            $g\epsilon^{2}$ & $\rho(g)\epsilon v$ & $\rho(g)v$ & $\rho(g)\epsilon^{2}v$ & $\lambda^2gw$ & $\lambda\sigma(g^{2})w$\tabularnewline
            $z$ & $\rho(z)v$ & $\rho(z)\epsilon v$ & $\rho(z)\epsilon^2v$ & $\sigma(z)w$ & $\sigma(z)gw$\tabularnewline
        \end{tabular}
    \end{center}
\end{rem}

\section{The Classification Theorem}
\label{section:main}

Now we state the main theorem of the paper.
It provides the classification of a class of
finite-dimensional Nichols algebras of group type.
We list important data of these Nichols algebras
in Table \ref{tab:classification}.

\begin{thm}
    \label{thm:big}
    Let $G$ be a non-abelian group and $V$ and $W$ be two
    finite-dimensional absolutely simple Yetter-Drinfeld modules over $G$.
    Assume that $c_{W,V}c_{V,W}\ne \id _{V\otimes W}$ and that the support of
    $V\oplus W$ generates the group $G$. Then the following are equivalent:
		\begin{enumerate}
            \item The Nichols algebra $\NA(V\oplus W)$ is finite-dimensional. 
            \item The pair $(V,W)$ admits all reflections and the Weyl groupoid of
				$(V,W)$ is finite.
             \item Up to permutation of its entries, the pair $(V,W)$ is one of
                 the pairs of Examples \ref{exa:G2a},
                 \ref{exa:G2b}, \ref{exa:G4}, \ref{exa:T}, \ref{exa:Z32},
                 \ref{exa:Z31a}, and \ref{exa:Z31b} of
                 Section~\ref{section:examples}.
    \end{enumerate}
    In this case, the rank and the dimension of $\NA(V\oplus
    W)$ appear in Table \ref{tab:classification}.
\end{thm}

Theorem \ref{thm:big} will be proved in Section \ref{section:bigproof}.

\section{Preliminaries} 
\label{section:preliminaries}

Let us first state some useful results from \cite{MR2734956,MR2732989}.  Recall
that $S_n\in \End (V^{\otimes n})$, where $n\in \N $, denotes the quantum
symmetrizer. 


\begin{lem}{\cite[Thm.~1.1]{MR2732989}}
  \label{lem:X_n}
  Let $V$ and $W$ be Yetter-Drinfeld modules over a Hopf algebra $H$ with
  bijective antipode.  Let $\varphi _0=0$, $X_0^{V,W}=W$, and 
  \begin{align*}
        \varphi_m &= \id-c_{V^{\otimes(m-1)}\otimes W,V}\,c_{V,V^{\otimes(m-1)}
			\otimes W}+(\id\otimes\varphi_{m-1})c_{1,2}\in \End (V^{\otimes m}\otimes
			W),\\
			X_m^{V,W} &= \varphi_m(V\otimes X_{m-1})\subseteq V^{\otimes m}\otimes W
  \end{align*}
  for all $m\geq1$. Then
$(\ad V)^n(W)\simeq X_n^{V,W}$ for all $n\in \N _0$.
\end{lem}

Now we collect information on $\Gamma _3$ which will be needed in
the paper.  By \cite[\S3]{MR2732989}, the center of $\Gamma_3$ is
$Z(\Gamma_3)=\langle\zeta,\gamma^2\rangle$, and the conjugacy classes of
$\Gamma_3$ are \begin{align} \label{eq:classes} \delta^G = \{\delta\},\quad
    (\nu \delta)^G=\{\nu \delta,\nu^2\delta\},\quad (\gamma
    \delta)^G=\{\nu^j\gamma\delta\mid 0\leq j\leq2\}, \end{align} where $\delta
    $ is running over all elements of $Z(\Gamma_3)$.  
    In this section,
    let $G$ be a group. Assume that there exist $g,\epsilon,z\in G$, $\epsilon
    \ne1$, such that there is a group epimorphism $\Gamma_3\to G$ with
    $\gamma\mapsto g$, $\nu\mapsto\epsilon$ and $\zeta\mapsto z$.  Note that
    the condition $\epsilon \not=1$ just means that $G$ is non-abelian.


%

\begin{lem} 
    \label{lem:simpleG3modules}
    Assume that $\K$ is algebraically closed. Let $V$ be a simple
    $\K G$-module and let $\rho :\K G \to \End (V)$ be the corresponding
    representation of $\K G$. Then $\dim V\le 2$. Moreover, if $\dim V=2$ then $\charK
    \not=3$, $\rho (1+\epsilon +\epsilon ^2)=0$, and the
    isomorphism class of $V$ is uniquely determined by the scalars $\rho (g^2)$
    and $\rho (z)$.
\end{lem}

\begin{proof}
    Let $v\in V$ be an eigenvector of $\rho (\epsilon )$.  Then $V=\K v+\K gv$
    and hence $\dim V\le 2$.

    Assume that $\dim V=2$. Then $gv\notin \K v$. If $\epsilon v=1$, then
    $\epsilon gv=g\epsilon ^2v=gv$. Then $\rho (\epsilon )=\id _V$, and
    $\K w$ for an eigenvector $w$ of $\rho (g)$ is a
    $\K G$-invariant subspace of $V$. This is a contradiction to the simplicity
    of $V$.
    Since $\epsilon ^3=1$ and $1+\epsilon +\epsilon ^2\in Z(\K G)$, we conclude
    that $\rho (1+\epsilon +\epsilon ^2)=0$ and $\charK\not=3$.
    Thus $v,gv$ is a basis of $V$ consisting of eigenvectors of $\rho (\epsilon
    )$.

    Let now $W$ be a simple $\K G $-module with $\dim W=2$, and let $w\in
    W\setminus \{0\}$ and $\lambda \in \K $ such that $\epsilon v=\lambda v$,
    $\epsilon w=\lambda w$. Assume that $g^2w=\rho (g^2)w$ and $zw=\rho (z)w$.
    Then the map $f:V\to W$, $v\mapsto w$, $gv\mapsto gw$, is an isomorphism of
    $\K G $-modules. This proves the lemma.
\end{proof}

\section{Reflections of the first pair}
\label{section:1}

Let $G$ be a non-abelian group, and let $g,\epsilon ,z\in G$. Assume that
there is an epimorphism $\Gamma_3\to G$ with $\gamma \mapsto g$,
$\nu \mapsto \epsilon $, $\zeta \mapsto z$.
Let $V\in \ydG $ such that $V\simeq M(g,\rho)$, where $\rho$ is an absolutely
irreducible representation of the centralizer $G^g=\langle z,g\rangle$. Since
this centralizer is abelian, $\deg\rho=1$. Let $v\in V_g$ with $v\ne0$. The
elements $v$, $\epsilon v$, $\epsilon^2 v$ form a basis of $V$. The degrees of
these basis vectors are $g$, $g\epsilon $ and $g\epsilon ^2$, respectively. 

\begin{rem}
	\label{rem:1:G_on_V}
	The action of $G$ on $V$ is given by the following table:
	\begin{center}
		\begin{tabular}{c|ccc}
			$M(g,\rho)$ & $v$ & $\epsilon v$ & $\epsilon^2 v$ \tabularnewline
			\hline
			$\epsilon$ & $\epsilon v$ & $\epsilon^2 v$ & $v$\tabularnewline
			$z$ & $\rho(z)v$ & $\rho(z)\epsilon v$ & $\rho(z)\epsilon^2v$\tabularnewline
			$g$ & $\rho(g)v$ & $\rho(g)\epsilon^2v$ & $\rho(g)\epsilon v$\tabularnewline
		\end{tabular}
	\end{center}
\end{rem}

Let $W\in \ydG $ such that $W\simeq M(\epsilon z,\sigma )$,
where $\sigma$ is an absolutely irreducible
representation of $G^{\epsilon z}=G^\epsilon=\langle \epsilon,z,g^2\rangle$.
Since $G^\epsilon$ is abelian, $\deg\sigma=1$. Let $w\in W_{\epsilon z}$ with
$w\ne0$.  Then $w$, $gw$ is a basis of $W$. The degrees of these basis vectors
are $\epsilon z$ and $\epsilon^2z$, respectively.

\begin{rem}
    \label{rem:1:G_on_W}
    The action of $G$ on $W$ is given by the following table:
    \begin{center}
        \begin{tabular}{c|cc}
            $W$ & $w$ & $gw$\tabularnewline
            \hline
            $\epsilon$ & $\sigma(\epsilon)w$ & $\sigma(\epsilon)^2gw$ \tabularnewline
            $z$ &  $\sigma(z)w$ & $\sigma(z)gw$ \tabularnewline
            $g$ &  $gw$ & $\sigma(g^2)w$ \tabularnewline
        \end{tabular}
    \end{center}
\end{rem}

In order to calculate $R_1(V,W)$,
we first compute the modules $(\ad V)^n(W)$ for $n\in\N$. For
$n\in\N$ we write $X_n=X_n^{V,W}$ and $\varphi_n=\varphi_n^{V,W}$.

\begin{lem}
	\label{lem:1:X1}
	The Yetter-Drinfeld module $X_1^{V,W}$
  is absolutely simple if and only if $\rho(z)^2\sigma(\epsilon g^2)=1$.
	In this case, $X_1^{V,W}\simeq M(gz,\sigma_1)$, where $\sigma_1$ is the
  character of $G^{gz}=\langle g,z\rangle$ with
	\begin{equation*}
		\sigma_1(g)=-\rho(gz^{-1})\sigma(\epsilon),\quad \sigma_1(z)=\rho(z)\sigma(z).
	\end{equation*}
	Let $w'=\varphi_1(\epsilon ^2v\otimes w)$.
  Then $w'\in(V\otimes W)_{gz}$ is non-zero. Moreover, 
	$w'$, $\epsilon w'$ and $\epsilon^2 w'$
  form a basis of $X_1^{V,W}$. The degrees of these
	basis vectors are $gz$, $g\epsilon z$ and $g\epsilon ^2z$, respectively.
\end{lem}

\begin{proof}
  By \cite[Lemma\,1.7]{MR2732989},
  $X_1^{V,W}\simeq \K G\varphi_1(\epsilon ^2v\otimes w)$. Using the actions
	of $G$ on $V$ and $W$ we obtain that
	\begin{equation}
		\begin{aligned}
			\label{eq:1:x1}
      w'=(\id -c_{W,V}c_{V,W})(\epsilon^2v\otimes w)
			=\epsilon^2v\otimes w-\rho(z)\sigma(\epsilon)^2\epsilon v\otimes gw,
		\end{aligned}
	\end{equation}
	and hence $w'\in(V\otimes W)_{gz}$ is non-zero. We compute
	\begin{align*}
		gw' &= g\epsilon^2v\otimes gw-\sigma(\epsilon)^2\rho(z) g\epsilon
    v\otimes g^2w\\
		&=\rho(g)\epsilon v\otimes gw
    -\sigma(\epsilon)^2\rho(gz)\sigma(g^2)\epsilon^2v \otimes w.
	\end{align*}
  Since $V$ and $W$ are absolutely simple and
  $G^{gz}=\langle g\rangle Z(G)$, the Yetter-Drinfeld module $X_1^{V,W}$
  is absolutely simple if and only if $gw'\in \K w'$. By the above
  calculations, this is equivalent to $\rho(z)^2\sigma(\epsilon g^2)=1$,
	and in this case $gw'=-\rho(gz^{-1})\sigma(\epsilon)w'$. The remaining
  claims are clear.
\end{proof}

Now we compute $X_2^{V,W}$. Since $(g,gz)$ and $(g\epsilon ^2,gz)$
represent the two orbits of $g^G\times (gz)^G$ under the diagonal
action of $G$, we conclude that
\begin{align} \label{eq:1:X2:decomp}
	X_2^{V,W}=\varphi_2(V\otimes X_1^{V,W})=\K G\{\varphi_2(v\otimes w'),\varphi_2(\epsilon^2v\otimes w')\}.
\end{align}
For the computation of $\varphi _2(v\otimes w')$ and
$\varphi _2(\epsilon^2 v\otimes w')$
we need the following. 

\begin{lem}
	\label{lem:1:phi1}
	Assume that $\rho (z)^2\sigma (\epsilon g^2)=1$.
  Let $w'=\varphi _1(\epsilon ^2v\otimes w)$. Then
	\begin{align}
		\label{eq:1:phi1(v,w)}&\varphi_1(v\otimes w)
    =\sigma(\epsilon)^{-1}\epsilon w',\\
		\label{eq:1:phi1(v,gw)}&\varphi_1(v\otimes gw)=-\rho(z)^{-1}\epsilon^2 w',\\
		\label{eq:1:phi1(e2v,gw)}&\varphi_1(\epsilon^2 v\otimes gw)
    =-\rho(z)^{-1}\sigma(\epsilon)^{-1}\epsilon w'.
	\end{align}
\end{lem}

\begin{proof}
	We first prove Equation \eqref{eq:1:phi1(v,w)}. Since
	\begin{align*}
		w'&=\varphi_1(\epsilon^2v\otimes w)=\epsilon^2\varphi_1(v\otimes\epsilon
    w)=\epsilon^2\sigma(\epsilon)\varphi_1(v\otimes w),
	\end{align*}
	Equation \eqref{eq:1:phi1(v,w)} holds.
  Now act with $g$ on Equation \eqref{eq:1:phi1(v,w)} and use
	Lemma \ref{lem:1:X1}
  to obtain Equation \eqref{eq:1:phi1(v,gw)}.
  Finally, to obtain Equation
	\eqref{eq:1:phi1(e2v,gw)} act with $\epsilon^2$ on Equation
	\eqref{eq:1:phi1(v,gw)}. 
\end{proof}

\begin{lem}
	\label{lem:1:phi2}
	Assume that $\rho (z)^2\sigma (\epsilon g^2)=1$.
	Then
	\begin{align}
		\label{eq:1:phi2(e2v,x1)}
    \varphi_2(\epsilon^2v\otimes w') = &\;
    (1+\rho(g))(\epsilon^2v\otimes
    w'+\rho(g)\sigma(\epsilon)v\otimes\epsilon w'),\\
		\label{eq:1:phi2(v,x1)}
		\varphi_2(v\otimes w')= &\;
      (1+\rho(g)^2\sigma(\epsilon))v\otimes w'
      +\rho(g)\sigma(\epsilon)^2\epsilon v\otimes \epsilon w'\\
			&\phantom{(1+\rho(g)^2\sigma(\epsilon))v\otimes w'\;}
      +\rho(g)\sigma(\epsilon)^2 \epsilon^2 v\otimes \epsilon^2 w'.
      \notag
	\end{align}
	In particular,
  $\varphi_2(\epsilon^2v\otimes w')=0$ if and only if $\rho(g)=-1$. Let
	$w''=\varphi_2(v\otimes w')$. Then $w''\in (V\otimes V\otimes W)_{g^2z}$
  is non-zero. 
\end{lem}

\begin{proof}
  Recall that
	$\varphi_2=\id-c_{X_1^{V,W},V}c_{V,X_1^{V,W}}+(\id\otimes\varphi_1)c_{1,2}$.
  Thus Lemma~\ref{lem:1:X1} and Equation \eqref{eq:1:x1} imply that
	\begin{multline*}
		\varphi_2(\epsilon^2v\otimes w') = \epsilon^2v\otimes w'
    +\rho(g)^2\sigma(\epsilon) v\otimes\epsilon w'\\
		+\rho(g)\epsilon^2v\otimes\varphi_1(\epsilon^2v\otimes w)
		-\rho(gz)\sigma(\epsilon)^2 v\otimes\varphi_1(\epsilon^2v\otimes gw).
	\end{multline*}
	Hence Equation \eqref{eq:1:phi2(e2v,x1)} follows from Lemma
	\ref{lem:1:phi1}. Similarly, we compute
	\begin{align*}
		\varphi_2(v\otimes w')=&\;(1+\rho(g)^2\sigma(\epsilon))v\otimes w'\\
		&+\rho(g)\epsilon v\otimes\varphi_1(v\otimes w)
		-\rho(gz)\sigma(\epsilon)^2 \epsilon^2v\otimes\varphi_1(v\otimes gw),
	\end{align*}
	and then use Lemma \ref{lem:1:phi1} to obtain Equation
	\eqref{eq:1:phi2(v,x1)}. From this the lemma follows.
\end{proof}

\begin{lem}
	\label{lem:1:X2}
	Assume that $\rho (z)^2\sigma (\epsilon g^2)=1$.
  Let $w''_1=\varphi _2(v\otimes w'_1)$.
	Then $X_2^{V,W}$ is absolutely simple if and only if $\rho(g)=-1$.
  In this case one of the following holds.
	\begin{enumerate}
		\item If $\sigma(\epsilon)^2+\sigma(\epsilon)+1=0$, then
      $X_2^{V,W}=\K w''_1\simeq M(g^2z,\sigma_2)$, where 
			$\sigma_2$ is the character	of $G$ given by
			\begin{align*}
				\sigma_2(g)=-\rho(z)^{-1}\sigma(\epsilon),\quad
				\sigma_2(\epsilon)=1,\quad
				\sigma_2(z)=\rho(z)^2\sigma(z).
			\end{align*}
		\item If $\sigma(\epsilon)=1$ and $\charK\ne3$, then
      $X_2^{V,W}\simeq M(g^2z,\sigma_2)$, where 
			$\sigma_2$ is the two-dimensional absolutely irreducible representation
			of $G$ with basis $\{w'',\epsilon w''\}$ and 
			\begin{align*}
				gw''=-\rho(z)^{-1}w'',\quad
				\epsilon^2 w''=-w''-\epsilon w'',\quad
				zw'' = \rho(z)^2\sigma(z)w''.  
			\end{align*}
	\end{enumerate}
\end{lem}

\begin{proof}
	Since $z\in Z(G)$, we conclude that $zw''=\rho(z)^2\sigma(z)w''$.
  Further, 	
	\[
		gw''=\varphi_2(gv\otimes gw')
    =-\rho(g^2z^{-1})\sigma(\epsilon)w''
	\]
  by Lemma~\ref{lem:1:X1}.
	Since $g^2z$ and $\epsilon g^2z$
	are not conjugate in $G$,
  Equation~\eqref{eq:1:X2:decomp} and
  Lemma~\ref{lem:1:phi2} imply that
  $X_2^{V,W}$ is absolutely simple if and
	only if $\rho(g)=-1$ and $\K Gw''$ is absolutely simple.
	
	Assume that $\rho (g)=-1$.
  Since $\epsilon^3=1$ in $G$, we know that $\sigma(\epsilon)^3=1$.
  Hence	$\sigma(\epsilon)^2+\sigma(\epsilon)+1=0$ or
  $\sigma(\epsilon)=1$.
	Using Equation \eqref{eq:1:phi2(v,x1)} and Lemma \ref{lem:1:X1} one directly
	computes 
	\begin{align}
		\label{eq:(1-e)x2}
		(1-\epsilon)w'' &= 
		(1+\sigma(\epsilon)+\sigma(\epsilon)^2)
    (v\otimes w'-\epsilon v\otimes\epsilon w'),
	\end{align}
	and similarly
	\begin{equation}
		\label{eq:(1+e+e2)x2}
			(1+\epsilon+\epsilon^2)w'' = (1-\sigma(\epsilon)^2)^2
      (1+\epsilon +\epsilon ^2)(v\otimes w').
	\end{equation}

	Suppose first that $\sigma(\epsilon)^2+\sigma(\epsilon)+1=0$. Then
	Equation \eqref{eq:(1-e)x2} becomes $(1-\epsilon)w''=0$ and
	hence the claim follows. 

	Suppose now that $\sigma(\epsilon)=1$ and $\charK\ne3$. Then
  $\{w'',\epsilon w''\}$ is linearly independent and
  $(1+\epsilon+\epsilon^2)w''=0$ by
	Equation \eqref{eq:(1+e+e2)x2}. Hence
	the lemma follows.
\end{proof}

\begin{lem}
	\label{lem:1:X3}
	Assume that $\rho(z)^2\sigma(\epsilon g^2)=1$, $\rho(g)=-1$.  Then
	$X_3^{V,W}=0$.
\end{lem}

\begin{proof}
  Lemma~\ref{lem:1:X2} implies that
  $X_3^{V,W}=\K G\varphi_3(\epsilon v\otimes X_2^{V,W})$. Now observe that
  $\epsilon ^2g(\epsilon v\otimes w'')\in \K \epsilon v\otimes \epsilon
  ^2w''$, and hence it is
	enough to prove that $\varphi _3(\epsilon v\otimes w'')=0$.

  Since $\rho(g)=-1$,
	Equation \eqref{eq:1:phi2(e2v,x1)} implies that
  $\varphi_2(\epsilon^2v\otimes w')=0$.
  Act on this equation by $g$ and by $\epsilon^2$ and use Lemma~\ref{lem:1:X1}
  to conclude	that $\varphi_2(\epsilon v\otimes w')=0$ and
  $\varphi_2(\epsilon v\otimes\epsilon ^2w')=0$. A direct calculation using
  Lemmas~\ref{lem:1:phi2} and \ref{lem:1:X2} then shows that
	\begin{align*}
			\varphi_3(\epsilon v\otimes w'')=&\;\epsilon v\otimes w''
      +\sigma (\epsilon )\epsilon v\otimes \epsilon ^2 w''
			+(1+\sigma(\epsilon)) g\epsilon v\otimes\varphi_2(\epsilon v\otimes w')\\
			&\quad 
      -\sigma(\epsilon)^2g\epsilon^2v\otimes\varphi_2(\epsilon v\otimes\epsilon w')
			-\sigma(\epsilon)^2gv\otimes\varphi_2(\epsilon v\otimes\epsilon^2
      w')\\
      =&\;\epsilon v\otimes (w''+\sigma (\epsilon )\epsilon ^2w''
      +\sigma (\epsilon )^2\epsilon w'').
	\end{align*}
  By Lemma~\ref{lem:1:X2}, if $\sigma (\epsilon )^2+\sigma (\epsilon )+1=0$
  then $\epsilon w''=w''$, and otherwise $\sigma (\epsilon )=1$ and
  $(1+\epsilon +\epsilon ^2)w''=0$.
	Hence $\varphi_3(\epsilon v\otimes w'')=0$. 
\end{proof}

Now we compute the modules $(\ad W)^n(V)$ for $n\in\N$. For $n\in\N$ let
$\varphi_n=\varphi_n^{W,V}$.

Let $X_1^{W,V}=\varphi_1(W\otimes V)$.  By \cite[Lemma\,1.7]{MR2732989},
$X_1^{W,V}=\K G\varphi_1(w\otimes\epsilon v)$.
Let $v_1'=\varphi_1(w\otimes\epsilon v)$.  A direct calculation yields 
\begin{equation} \label{eq:v1'}
	v_1'=w\otimes\epsilon v-\rho(z)\sigma(\epsilon ^2) gw\otimes\epsilon^2 v.
\end{equation}
Hence $v_1'\in(W\otimes V)_{gz}$ is non-zero. 

\begin{lem}
	\label{lem:1:Y1}
	Assume that $\rho(z)^2\sigma(\epsilon g^2)=1$. 
	Then $X_1^{W,V}$ is an absolutely simple Yetter-Drinfeld module, and
  $X_1^{W,V}\simeq M(gz,\rho_1)$, where
	$\rho_1=\sigma_1$. Moreover, $v_1'$, $\epsilon v_1'$, $\epsilon^2 v_1'$
  is a basis of $X_1^{W,V}$, and the degrees of these basis vectors are $gz$,
	$g\epsilon z$, and $g\epsilon ^2z$, respectively. 
\end{lem}

\begin{proof}
  Since $X_1^{V,W}\simeq X_1^{W,V}$ via $c_{V,W}$, the claim follows from
  Lemma~\ref{lem:1:X1}.
\end{proof}

\begin{lem}
	\label{lem:1:Y2}
	Assume that $\rho(z)^2\sigma(\epsilon g^2)=1$ and $\rho(g)=-1$.
  Then $X_2^{W,V}=0$
  if and only if $\sigma(\epsilon z)=-1$. Moreover, $X_2^{W,V}$ is
	absolutely simple if and only if $\sigma(\epsilon z)\ne-1$ and
	$\sigma(\epsilon z^2)=1$.  In this case $X_2^{W,V}\simeq M(gz^2,\rho_2)$,
	where $\rho_2$ is the character of $G^g$ given by 
	\begin{align*}
    \rho_2(g)=-\rho(z)^{-2}\sigma(z)^{-1},\quad
		\rho_2(z)=\rho(z)\sigma(z)^2.
	\end{align*}
\end{lem}

\begin{proof}
	First we act with $\epsilon$ and $\epsilon^2$ on
	$v_1'=\varphi_1(w\otimes\epsilon v)$ to obtain the following formulas:
	\begin{align}
		\label{eq:1:phi1(w,v),phi1(w,e2v)}
		\varphi_1(w\otimes\epsilon^2v) = \sigma(\epsilon)^2\epsilon v_1', \quad
		\varphi_1(w\otimes v) = \sigma(\epsilon)\epsilon^2 v_1'.
	\end{align}
	Since $G$ acts transitively on $(\epsilon z)^G\times (gz)^G$ via the diagonal
  action, we obtain from
  \cite[Lemma\,1.7]{MR2732989} that $X_2^{W,V}=\K G\varphi _2(w\otimes \epsilon
  v_1')$. We compute
	\begin{multline*}
		\varphi_2(w\otimes\epsilon v_1')
		=w\otimes\epsilon v_1'-\rho(z)\sigma(\epsilon ^2z^2)gw\otimes\epsilon^2v_1'\\
		+\sigma(\epsilon ^2z) w\otimes\varphi_1(w\otimes\epsilon^2v)
    -\rho(z)\sigma(z)gw\otimes\varphi_1(w\otimes v).
	\end{multline*}
	Equations \eqref{eq:1:phi1(w,v),phi1(w,e2v)} imply that
	\begin{align}
		\varphi_2(w\otimes\epsilon v_1')=(1+\sigma(\epsilon z))
    (w\otimes\epsilon v_1'-\rho(z)\sigma(\epsilon z)gw\otimes\epsilon^2v_1').
	\end{align}
	Hence $\varphi_2(w\otimes\epsilon v_1')=0$ if and only if
  $\sigma(\epsilon z)=-1$. 
	Let 
	\[
		v_1''= w\otimes\epsilon v_1'-\rho(z)\sigma(\epsilon z)gw\otimes\epsilon^2v_1'. 
	\]
	Then $v_1''\in(W\otimes X_1^{W,V})_{gz^2}$ is non-zero.
  Since $\rho(z)^2\sigma(\epsilon g^2)=1$, 
	\begin{align*}
		gv_1'' &= gw\otimes\epsilon^2gv_1'-\rho(z)\sigma(\epsilon z)
    \sigma(g)^2w\otimes\epsilon gv_1'\\
		&=gw\otimes\rho_1(g)\epsilon^2 v_1'
    -\rho(z)^{-1}\sigma(z)\rho_1(g)w\otimes\epsilon v_1'\\
		&=-\rho_1(g)\rho(z)^{-1}\sigma(z)(w\otimes\epsilon v_1'
    -\rho(z)\sigma(z)^{-1}gw\otimes\epsilon^2 v_1').
	\end{align*}
	Thus $X_2^{W,V}$ is absolutely simple if and only if
	$\sigma(\epsilon z)\ne1$ and $gv_1''\in \K v_1''$. This
	is equivalent to $\sigma(\epsilon z)\ne1$ and
  $\sigma(\epsilon z^2)=1$. 
	Finally, the equation $zv_1'=\rho(z)\sigma(z)^2v_1'$
  follows from $v_1'\in W\otimes W\otimes V$ and $z\in Z(G)$. 
\end{proof}

\begin{lem}
	\label{lem:1:Yn_auxiliar}
	Assume that $\rho(z)^2\sigma(\epsilon g^2)=1$, $\rho(g)=-1$,
	$\sigma(\epsilon z)\ne-1$ and $\sigma(\epsilon z^2)=1$.  
  Define inductively $y_0=v$ and 
	\begin{align*}
    y_n=w\otimes\epsilon y_{n-1}-\rho(z)\sigma(z^{3n-1})gw\otimes\epsilon^2 y_{n-1}
	\end{align*}
	for all $n\geq1$. 
  Then $y_{n}\in(W^{\otimes n}\otimes V)_{gz^n}$ and 
	\begin{align}
		\label{eq:1:phin}
		\varphi_{n}(w\otimes\epsilon y_{n-1})=
    (1+\sigma(z)^{-1}\cdots+\sigma(z)^{-n+1})y_{n}
	\end{align}
	for all $n\geq 1$. 
\end{lem}

\begin{proof}
	We proceed by induction on $n$. For $n=1$ the claim holds by
  Equation~\eqref{eq:v1'} and since $\sigma (\epsilon )^2=\sigma (\epsilon
  )^{-1}=\sigma (z^2)$. It is also clear that $y_n\in (W^{\otimes n}\otimes
  V)_{gz^n}$ for all $n\ge 0$.
  
  Assume
  that Equation~\eqref{eq:1:phin} holds for some $n\geq1$.
  Apply $\epsilon$ and $\epsilon^2$ to
	Equation \eqref{eq:1:phin} to obtain 
	\begin{align}
		\label{eq:1:phin(w,e2yn-1)}\varphi_n(w\otimes\epsilon^2 y_{n-1})=&\;
    (1+\sigma(z)^{-1}+\cdots+\sigma(z)^{-n+1})\sigma(\epsilon)^2\epsilon y_n,\\
		\label{eq:1:phin(w,eyn-1)}\varphi_n(w\otimes y_{n-1})=&\;
    (1+\sigma(z)^{-1}+\cdots+\sigma(z)^{-n+1})\sigma(\epsilon)\epsilon^2 y_n.
	\end{align}
	Since $\sigma(\epsilon z^2)=1$, 
	\begin{multline*}
		\varphi_{n+1}(w\otimes\epsilon y_n) = w\otimes\epsilon
    y_n-\rho(z)\sigma(z)^{2n}\sigma(\epsilon)^2 gw\otimes\epsilon^2 y_n\\
    +\sigma(\epsilon ^2z) w\otimes\varphi_n(w\otimes\epsilon^2y_{n-1})
    -\rho(z)\sigma(z^{3n-1})\sigma(\epsilon z)gw\otimes\varphi_n(w\otimes y_{n-1}).
	\end{multline*}
	Using Equations \eqref{eq:1:phin(w,e2yn-1)} and \eqref{eq:1:phin(w,eyn-1)} and
	$\sigma(\epsilon z^2)=1$ one obtains 
	\[
		\varphi_{n+1}(w\otimes\epsilon y_n)=(1+\sigma(z)^{-1}+\cdots+\sigma(z)^{-n})y_{n+1},
	\]
	as desired.
\end{proof}

\begin{lem}
	\label{lem:1:Yn}
	Assume that $\rho(z)^2\sigma(\epsilon g^2)=1$, $\rho(g)=-1$,
	$\sigma(\epsilon z)\ne-1$, and $\sigma(\epsilon z^2)=1$.
	Then $X_n^{W,V}\simeq M(gz^n,\rho_n)$ for all $n\ge 1$,
  where $\rho_n$ is the character of $G^g$ given by
	\begin{align*}
	  \rho_n(g)=(-1)^{n+1}\rho(z)^{-n}\sigma(z)^{(3n+5)n/2},\quad
	  \rho_n(z)=\rho(z)\sigma(z)^n.
	\end{align*}
	Moreover, $X^{W,V}_n=0$ if and only if $(n)^!_{\sigma (z)}=0$.
\end{lem}

\begin{proof}
  For all $n\ge 0$ let $y_n$ be as in Lemma~\ref{lem:1:Yn_auxiliar}.
  Then $zy_n=\rho(z)\sigma(z)^ny_n$, since $y_n\in W^{\otimes n}\otimes V$ and
  $z\in Z(G)$. To
	prove that $gy_n=\rho_n(g)y_n$ for all $n\ge 1$ we proceed by induction.
  For $n=1$ this equation holds by Lemmas~\ref{lem:1:Y1} and \ref{lem:1:X1}.
  Suppose now that
	$gy_n=\rho_n(g)y_n$ for some $n\geq1$. Then 
	\begin{align*}
		gy_{n+1} &= gw\otimes\epsilon^2gy_n
    -\rho(z)\sigma(z^{3n+2})\sigma(g^2)w\otimes\epsilon gy_n\\
		&=\rho_n(g)(gw\otimes\epsilon^2y_n
    -\rho(z)\sigma(z^{3n+2})\sigma(g^2) w\otimes\epsilon y_n).
	\end{align*}
	Since $\sigma(\epsilon z^2)=1$ and
	$\rho(z)^2\sigma(\epsilon g^2)=1$, we conlude that the expression for
	$\rho_{n+1}(g)$ given in the claim can be written as 
	\[
  \rho_{n+1}(g)=-\rho_n(g)\rho(z)^{-1}\sigma(z^{3n-2}).
	\]
	This implies the claim.

	To complete the proof of the lemma we observe that 
	\[
	X^{W,V}_n=\K G\varphi_n(w\otimes\epsilon y_n)
	\]
	and hence Equation \eqref{eq:1:phin} of Lemma \ref{lem:1:Yn_auxiliar} implies
  that $X^{W,V}_n=0$ if and only if $(k)_{\sigma (z)}=0$ for
	some $k\le n$, which is equivalent to say that $(n)^!_{\sigma (z)}=0$.
\end{proof}

We collect all the results of this section in the following proposition. 
We write $a_{i,j}^{(V,W)}$ for the entries of the Cartan matrix of the pair $(V,W)$.

\begin{pro}
	\label{pro:1}
	Let $V,W\in \ydG $ such that $V\simeq M(g,\rho)$,
        where $\rho$ is a character of $G^g$,
        and $W\simeq M(\epsilon z,\sigma)$,
        where $\sigma$ is a character of $G^{\epsilon z}$. Then the following hold:
	\begin{enumerate}
		\item $(\ad V)^m(W)$ and $(\ad W)^m(V)$ are absolutely simple or zero for
			all $m\in\N_0$ if and only if 
			\begin{equation*}
				\rho(z)^2\sigma(\epsilon g^2)=1,\quad
				\rho(g)=-1,\quad
				(\sigma(\epsilon z)+1)(\sigma(\epsilon z^2)-1)=0.
			\end{equation*}
			In particular, these equations imply that $\sigma(z)^6=1$. 
		\item Assume that the equations for $\rho $ and $\sigma $ in (1) hold.
            Then the Cartan matrix of $(V,W)$ satisfies
      $a_{1,2}^{(V,W)}=-2$ and $X_2^{V,W}\simeq
			M(g^2z,\sigma_2)$, where
			\begin{enumerate}
				\item if $1+\sigma(\epsilon)+\sigma(\epsilon)^2=0$, then 
					$\sigma_2$ is the character of $G$ given
					by $\sigma_2(\epsilon)=1$,
					$\sigma_2(g)=-\rho(z)^{-1}\sigma(\epsilon)$,
					$\sigma_2(z)=\rho(z)^2\sigma(z)$, and
				\item if $\sigma(\epsilon)=1$ and $\charK\ne3$, then 
					$\sigma_2$ is a two-dimensional absolutely irreducible representation
					of $G$ with
					$\sigma_2(g^2)=\rho(z)^{-2}$
          and  $\sigma_2(z)=\rho(z)^2\sigma(z)$.
			\end{enumerate}
            Moreover, 
			\begin{equation*}
				a_{2,1}^{(V,W)}=\begin{cases}
					-1 & \text{if $\sigma(\epsilon z)=-1$},\\
					-2 & \text{if $\sigma(z)=\sigma(\epsilon)$ and $(3)_{\sigma(\epsilon)}=0$},\\
					-5 & \text{if $\sigma(z)+\sigma(\epsilon)=(3)_{\sigma(\epsilon)}=0$ and $\charK\ne2,3$},\\
					1-p & \text{if $\charK=p\ge 5$ and $\sigma(\epsilon)=\sigma(z)=1$},
				\end{cases}
			\end{equation*}
			and otherwise $(\ad W)^m(V)\neq0$ for all $m\in\N_0$. In these cases
			$X_m^{W,V}\simeq M(gz^m,\rho_m)$ for $m=-a_{2,1}^{(V,W)}$, where $\rho_m$
			is the character of $G^g$ with 
			$\rho_1(z)=\rho(z)\sigma(z)$, $\rho_1(g)=\rho(z)^{-1}\sigma(\epsilon)$ and 
			\begin{align*}
				\rho_m(z)=\rho(z)\sigma(z)^m,
			  	\quad \rho_m(g)=(-1)^{m+1}\rho(z)^{-m}\sigma(z)^{(3m+5)m/2}
			\end{align*}
			for all $m\geq2$.
	\end{enumerate}
\end{pro}

\begin{proof}
	The first claim follows from Lemmas \ref{lem:1:X1}, \ref{lem:1:X2},
	\ref{lem:1:X3}, \ref{lem:1:Y1}, \ref{lem:1:Y2} and \ref{lem:1:Yn}. Further,
	Lemmas \ref{lem:1:X2} and \ref{lem:1:X3}
        yield the claim concerning $a_{1,2}^{(V,W)}$. Now we
	prove the claim concerning $a_{2,1}^{(V,W)}$. Since $(\sigma(\epsilon
	z)+1)(\sigma(\epsilon z^2)-1)=0$ and $\sigma(\epsilon)^3=1$, we need to
	consider the following cases:
	\begin{enumerate}
		\item $\sigma(\epsilon z)=-1$ and $\sigma(\epsilon)^3=1$, 
		\item $\sigma(\epsilon)=\sigma(z)$ and $\sigma(\epsilon)^2+\sigma(\epsilon)+1=0$,
		\item $\sigma(\epsilon)=\sigma(z)=1$ and $\charK=p\ge 5$,
		\item $\sigma(\epsilon)=\sigma(z)=1$ and $\charK=0$,
		\item $\sigma(\epsilon)=-\sigma(z)$, $\sigma(\epsilon)^2+\sigma(\epsilon)+1=0$ and $\charK\ne2,3$.
	\end{enumerate}
    Using the equivalence between $(\ad W)^m(V)=0$ and
    $(m)^!_{\sigma(z)}=0$ of Lemma \ref{lem:1:Yn}, one easily obtains the
    second claim.
\end{proof}

\begin{cor}
	\label{cor:R1:A}
	Let $V,W\in \ydG $ with $V\simeq M(g,\rho)$ and $W\simeq M(\epsilon z,\sigma)$,
	where $\rho$ is a character of $G^g$ and $\sigma$ is a character of
	$G^{\epsilon }$. Assume that 
	\[
	\rho(g)=\sigma(\epsilon z)=-1,\quad\rho(z)^2\sigma(\epsilon g^2)=1,
  \quad (3)_{\sigma(\epsilon)}=0.
	\]
	Let $g'=g^{-1}$, $\epsilon'=\epsilon $, $z'=g^2z$, $\rho'$ be the
	representation of $G^{g'}$ dual to $\rho$, and $\sigma'$ be the character of
	$G$ given by $\sigma'(g)=-\rho(z)^{-1}\sigma(\epsilon)$,
	$\sigma'(\epsilon)=1$, $\sigma'(z)=\rho(z)^2\sigma(z)$. Then 
  $a_{1,2}^{(V,W)}=-2$ and
	\[
	R_1(V,W)=\left(V^*,X_2^{V,W}\right)
	\]
	with $V^*\simeq M(g',\rho')$, $X_2^{V,W}\simeq M(z',\sigma')$ and 
	\begin{align*}
	  \rho'(g')=-1,\quad
	  \rho'(z')\sigma'(z'g')=1,\quad
	  1-\rho'(z')\sigma'(g')+\rho'(z')^2\sigma'(g')^2=0.
	\end{align*}
\end{cor}

\begin{proof}
	Using Proposition \ref{pro:1} one obtains that $a_{1,2}^{(V,W)}=-2$ and hence
	the description of $R_1(V,W)$ follows.  It is clear that
	$\rho'(g')=\rho(g)=-1$. A direct calculation yields 
	\begin{align*}
		&\rho'(z')\sigma'(z'g')=-\sigma(\epsilon z)=1,\\
		&1-\rho'(z')\sigma'(g')+\rho'(z')^2\sigma'(g')^2=1+\sigma(\epsilon)^{-1}+\sigma(\epsilon)^{-2}=0.
	\end{align*}
	This completes the proof.
\end{proof}

\begin{cor}
	\label{cor:R2:A}
	Let $V,W\in \ydG $ with $V\simeq M(g,\rho)$ and $W\simeq M(\epsilon z,\sigma)$,
	where $\rho$ is a character of $G^g$ and $\sigma$ is a character of
	$G^{\epsilon }$. Assume that 
	\[
		\rho(g)=\sigma(\epsilon z)=-1,\quad\rho(z)^2\sigma(\epsilon g^2)=1,
    \quad (3)_{\sigma(\epsilon)}=0.
	\]
	Let $g''=gz$, $\epsilon''=\epsilon^{-1}$, $z''=z^{-1}$, $\rho''$ be
	the character of $G^g$ given by
	$\rho''(g)=\rho(z)^{-1}\sigma(\epsilon)$ and
	$\rho''(z)=\rho(z)\sigma(z)$, and $\sigma''$ be the representation of
	$G^{\epsilon z}$ dual to $\sigma$. Then 
  $a_{2,1}^{(V,W)}=-1$ and
	\[
			R_2(V,W)=\left(X_1^{W,V},W^*\right)
	\]
	with $X_1^{W,V}\simeq M(g'',\rho'')$, 
	$W^*\simeq M(\epsilon ''z'',\sigma'')$, and  		
	\begin{align*}
    \rho''(g'')=\sigma''(\epsilon''z'')=-1,\quad
    \rho''(z'')^2\sigma''(\epsilon''g''^2)=1,\quad
    (3)_{\sigma''(\epsilon'')}=0.
	\end{align*}
\end{cor}

\begin{proof}
	The assumptions on $\rho$ and $\sigma$ and Proposition \ref{pro:1} yield
	$a_{2,1}^{(V,W)}=-1$ and hence the description of $R_2(V,W)$ follows. Then we compute 
	\begin{align*}
		&\sigma''(\epsilon''z'')=\sigma''(\epsilon^{-1}z^{-1})=\sigma(\epsilon z)=-1,\\
		&\rho''(g'')=\rho''(gz)=\rho''(g)\rho''(z)=-1,\\
		&1+\sigma''(\epsilon'')+\sigma''(\epsilon'')^2=1+\sigma(\epsilon)+\sigma(\epsilon)^2=0.
	\end{align*}
	Using that $g$ and $z$ commute, we obtain that 
	\begin{align*}
		\rho''(z'')^2\sigma''(\epsilon''g''^2)
    =\rho''(z)^{-2}\sigma''(\epsilon^{-1}g^2z^2)
		=\rho(z)^{-2}\sigma(\epsilon)\sigma(g)^{-2}\sigma(z)^{-4}.
	\end{align*}
	Since $\sigma(\epsilon z)=-1$ we obtain that $\sigma(\epsilon
	z^{-4})=\sigma(\epsilon)^{-1}$. Hence we conclude that 
	\[
        \rho''(z'')^2\sigma''(\epsilon''g''^2)=(\rho(z)^2\sigma(\epsilon g^2))^{-1}=1. 
    \]
    This completes the proof.
\end{proof}

\begin{cor}
	\label{cor:R1:B}
	Let $V,W\in \ydG $ with $V\simeq M(g,\rho)$ and $W\simeq M(\epsilon z,\sigma)$,
	where $\rho$ is a character of $G^g$ and $\sigma$ is a character of
	$G^{\epsilon }$.
	Assume that $\charK\ne3$ and
	\[
		\rho(g)=-1,\quad
		\rho(z^2)\sigma(\epsilon g^2)=1,\quad
		\sigma(\epsilon)=1,\quad
		\sigma(z)=-1.
	\]
	Further, let $g'=g^{-1}$, $\epsilon'=\epsilon $, $z'= g^2z$, let $\rho'$ be
	the irreducible representation of $G^g$ dual to $\rho$, and let $\sigma'$ be
  an absolutely irreducible representation of $G$ with $\deg\sigma'=2$,
	$\sigma'(g^2)=\rho(z)^{-2}$, and $\sigma'(z)=-\rho(z)^2$.  Then 
  $a_{1,2}^{(V,W)}=-2$ and
	\[
		R_1(V,W)=\left(V^*,X_2^{V,W}\right)
	\]
	with $V^*\simeq M(g',\rho')$, $X_2^{V,W}\simeq M(z',\sigma')$, and 
	\begin{align*}
	  \rho'(g')=-1,\quad \rho'(z')^2\sigma'(g'^2)=1,\quad \sigma'(z')=-1.
	\end{align*}
\end{cor}

\begin{proof}
	It is similar to the proof of Corollary \ref{cor:R1:A}.
\end{proof}

\begin{cor}
	\label{cor:R2:B}
	Let $V,W\in \ydG $ with $V\simeq M(g,\rho)$ and $W\simeq M(\epsilon z,\sigma)$,
	where $\rho$ is a character of $G^g$ and $\sigma$ is a character of
	$G^{\epsilon }$.
	Assume that $\charK\ne3$ and
	\[
		\rho(g)=-1,\quad
		\rho(z)^2\sigma(\epsilon g^2)=1,\quad 
		\sigma(\epsilon)=1,\quad
		\sigma(z)=-1.
	\]
	Further, let $g''=gz$, $\epsilon''=\epsilon^{-1}$, $z''= z^{-1}$, let
  $\rho''$ be the character of $G^g$ given by $\rho''(z)=-\rho(z)$ and
	$\rho''(g)=\rho(z)^{-1}$, and let $\sigma''$ be the character of
	$G^{\epsilon }$ dual to $\sigma$. Then 
  $a_{2,1}^{(V,W)}=-1$ and
	\[
		R_2(V,W)=\left(X_1^{W,V},W^*\right)
	\]
	with $X_1^{W,V}\simeq M(g'',\rho'')$, $W^*\simeq M(\epsilon''z'',\sigma'')$,
	and 
	\begin{align*}
    \rho''(g'')=-1,\quad \rho''(z'')^2\sigma''(\epsilon''g''^2)=1,\quad
		\sigma''(\epsilon'')=1,\quad \sigma''(z'')=-1.
	\end{align*}
\end{cor}

\begin{proof}
	It is similar to the proof of Corollary \ref{cor:R2:A}.
\end{proof}

\section{Reflections of the second pair}
\label{section:2}

In this section we have to deal with an irreducible representation of $G$ of
degree two. Therefore we assume that $\K $ is algebraically closed.
This will not be relevant for our classification of Nichols algebras.

Let $G$ be a non-abelian group,
and let $g,\epsilon ,z\in G$. Assume that there is an
epimorphism $\Gamma_3\to G$ with $\gamma \mapsto g$,
$\nu \mapsto \epsilon $, $\zeta \mapsto z$.
Let $V,W\in \ydG $ such that $V\simeq M(g,\rho )$ and $W\simeq M(z,\sigma)$, where
$\rho $ is a character of $G^g$ and
$\sigma$ is a two-dimensional irreducible representation of $G$.
Let $\lambda $ be an eigenvalue of $\sigma (\epsilon )$ and let
$w$ be a corresponding eigenvector.
Then, by Lemma~\ref{lem:simpleG3modules}, $\charK\neq3$ and
$1+\lambda +\lambda ^2=0$. Hence $\epsilon gw=\lambda ^{-1}gw$,
$\lambda ^{-1}\not=\lambda $, and
$\{w,gw\}$ is a basis of $W_z=W$. 

\begin{rem}
    \label{rem:2:G_on_W}
    By the above discussion, we obtain the following table for the action of $G$
    on $W$:
    \begin{center}
        \begin{tabular}{c|cc}
            $W$ & $w$ & $gw$\tabularnewline
            \hline
            $\epsilon $ & $\lambda w$ & $\lambda ^2 gw$ \tabularnewline
            $z$ &  $\sigma(z)w$ & $\sigma(z)gw$ \tabularnewline
            $g$ &  $gw$ & $\sigma(g^2)w$ \tabularnewline
        \end{tabular}
    \end{center}
\end{rem}

Now we compute the modules $(\ad V)^m(W)$ for $m\in\N$. First we obtain that
$X_1^{V,W}=\varphi_1(V\otimes W)=\K G\varphi_1(v\otimes w)$ and
\begin{equation}
	\label{eq:2:phi1(v,w)}
	\varphi_1(v\otimes w)=v\otimes(w-\rho(z)gw).
\end{equation}
Hence $w'\coloneqq\varphi_1(v\otimes w)\in(V\otimes W)_{gz}$ is non-zero. 

\begin{lem}
	\label{lem:2:X1}
	The Yetter-Drinfeld module
  $X_1^{V,W}$ is simple if and only if
  $\rho(z)^2\sigma(g^2)=1$.  In this case,
  $X_1^{V,W}\simeq M(gz,\sigma_1)$, where $\sigma _1$ is the character
	of $G^{g}$ with
	\begin{align*}
		\sigma_1(z)=\rho(z)\sigma(z),\quad
		\sigma_1(g)=-\rho(gz^{-1}).
	\end{align*}
	A basis for $X_1^{V,W}$ is given by $\{w',\epsilon w',\epsilon^2 w'\}$. The
	degrees of these basis vectors are $gz$, $g\epsilon z$ and $g\epsilon ^2z$,
	respectively. 
\end{lem}

\begin{proof}
  Since $X_1^{V,W}=\K Gw'$ and $w'\in (V\otimes W)_{gz}$,
  $X_1^{V,W}$ is simple if and only if
	$gw'=\K w'$. Since 
	\[
		gw'=gv\otimes (gw-\rho(z)g^2w)=\rho (g)v\otimes (-\rho (z)\sigma (g^2)w
    +gw),
	\]
	the latter is equivalent to $gw'=-\rho (gz)\sigma (g^2)w'$,
  $\rho (z)^2\sigma (g^2)=1$.
	From this the claim follows.
\end{proof}

\begin{rem}
	\label{rem:2:G_on_X1}
	The action of $G$ on $X_1^{V,W}$ can be displayed in a table similar to the
  one in Remark~\ref{rem:1:G_on_V}, where $v$ has to be replaced by $w'$ and
  $\rho $ has to be replaced by $\sigma _1$.
\end{rem}

\begin{lem}
	\label{lem:2:X2}
  Assume that $\rho(z)^2\sigma(g^2)=1$. Then $X_2^{V,W}\not=0$, and $X_2^{V,W}$
  is simple if and
	only $\rho(g)=-1$. In this case, $X_2^{V,W}\simeq M(\epsilon g^2z,\sigma_2)$, where
	$\sigma_2$ is the character of $G^{\epsilon}$ given by 
	\begin{align*}
		\sigma_2(\epsilon)=1,\quad
		\sigma_2(z)=\rho(z)^2\sigma(z),\quad
		\sigma_2(g^2)=\rho(z)^{-2}.
	\end{align*}
	The set $\{w'',gw''\}$ is a basis of $X_2^{V,W}$. The degrees of these vectors are
	$\epsilon g^2z$ and $\epsilon^2 g^2z$, respectively. 
\end{lem}

\begin{proof}
	Observe that $X_2^{V,W}$ is the direct sum
  $\K G\varphi_2(v\otimes w')\oplus \K G\varphi_2(\epsilon^2v\otimes w')$
  of two Yetter-Drinfeld submodules.
  Acting with $g$ on Equation
	\eqref{eq:2:phi1(v,w)} and using Lemma \ref{lem:2:X1} we obtain that
	\begin{equation}
		\label{eq:2:phi1(v,gw)}
		\varphi_1(v\otimes gw)=-\rho(z)^{-1}w'.
	\end{equation}
	A direct computation shows that
	\begin{align*}
		\varphi_2(v\otimes w') = v\otimes w'&+\rho(g^2)v\otimes w'\\
		&+\rho(g)v\otimes\varphi_1(v\otimes w)-\rho(gz)v\otimes\varphi_1(v\otimes gw).
	\end{align*}
	Then $\varphi_2(v\otimes w')=(1+\rho(g))^2v\otimes w'$. To compute
	$\varphi_2(\epsilon^2v\otimes w')$ we act with $\epsilon^2$ on Equations
	\eqref{eq:2:phi1(v,w)} and \eqref{eq:2:phi1(v,gw)} to obtain
	\begin{align*}
		\varphi_1(\epsilon^2v\otimes w)=\lambda \epsilon^2w',\quad
		\varphi_1(\epsilon^2v\otimes gw)=-\lambda ^2\rho(z)^{-1}\epsilon^2w'.
	\end{align*}
	Then one computes
	\begin{align*}
		\varphi_2(\epsilon^2v\otimes w')=\epsilon^2v\otimes w'+\rho(g)^2v\otimes \epsilon w'
		&+\rho(g)\epsilon v\otimes\varphi_1(\epsilon^2v\otimes w)\\
		&-\rho(z)\rho(g)\epsilon v\otimes\varphi_1(\epsilon^2v\otimes gw),
	\end{align*}
	and hence 
	\begin{align}
		\label{eq:2:phi2(e2v,x1)}
		\varphi_2(\epsilon^2v\otimes w')=\epsilon^2v\otimes w'
		+\rho(g^2)v\otimes\epsilon w'
		-\rho(g)\epsilon v\otimes\epsilon^2w'.
	\end{align}
	Therefore $w''\coloneqq\varphi_2(\epsilon^2v\otimes w')\in(V\otimes V\otimes
	W)_{\epsilon g^2z}$ is non-zero. We conclude that
  $\K G\varphi_2(v\otimes w')=0$, that is,
	$\rho(g)=-1$. In this case one easily obtains
  from Equation \eqref{eq:2:phi2(e2v,x1)}
  that
	$\epsilon w''=w''$. The equations
	$g^2w''=\rho(z)^{-2}w''$ and $zw''=\rho(z)^2\sigma(z)w''$ are clear
  since $g^2,z\in Z(G)$ and $\sigma(g^2)=\rho (z)^{-2}$.
\end{proof}

\begin{lem}
	\label{lem:2:X3}
	Assume that $\rho(z)^2\sigma(g^2)=1$ and $\rho(g)=-1$. Then $X_3^{V,W}=0$. 
\end{lem}

\begin{proof}
	Act with $\epsilon$ and with $g\epsilon $ on
  $w''=\varphi_2(\epsilon^2v\otimes w')$
  and use Lemmas~\ref{lem:2:X2} and \ref{lem:2:X1}
	to obtain that $w''=\varphi_2(v\otimes\epsilon w')$ and
	$gw''=-\rho(z)^{-1}\varphi_2(v\otimes\epsilon^2w')$, respectively.
	Now we calculate 
	\begin{align*}
		 \varphi_3(v\otimes w'')=v\otimes w''&-\rho(z)\epsilon^2 v\otimes gw''
		 -\epsilon v\otimes\varphi_2(v\otimes w')\\
		 &-v\otimes\varphi_2(v\otimes\epsilon w')-\epsilon^2v\otimes\varphi_2(v\otimes\epsilon^2 w'),
	\end{align*}
  and hence $\varphi_3(v\otimes w'')=0$. Since $X_3^{V,W}=\K G\varphi
  _3(v\otimes w'')$, the proof of the lemma is completed.
\end{proof}

Now we compute the modules $(\ad W)^n(V)$ for $n\in\N$. For $n\in\N$ we write
$\varphi_n=\varphi_n^{W,V}$.

First note that
\[
	X_1^{W,V}=\varphi_1(W\otimes V)=\K G\varphi_1(w\otimes v).
\]
Further,
\begin{equation}
	\label{eq:2:y1}
	v_1'\coloneqq\varphi_1(w\otimes v)=(w-\rho(z)gw)\otimes v\in(W\otimes V)_{gz}
\end{equation}
is non-zero. 

\begin{lem}
	\label{lem:2:Y1} 
	Assume that $\rho(z)^2\sigma(g^2)=1$. Then $X_1^{W,V}$ is simple and
	$X_1^{W,V}\simeq M(gz,\rho_1)$, where $\rho_1$ is the character of $G^{gz}$ defined
	by \[
		\rho_1(g)=-\rho(gz^{-1}),\quad
		\rho_1(z)=\rho(z)\sigma(z).
	\]
	The set $\{v_1',\epsilon v_1',\epsilon^2 v_1'\}$ is a basis of $X_1^{W,V}$. The
	degrees of these basis vectors are $gz$, $g\epsilon z$, and $g\epsilon ^2z$,
	respectively.
\end{lem}

\begin{proof}
	This follows from Lemma \ref{lem:2:X1}. 
\end{proof}

\begin{lem}
	\label{lem:2:Y2}
	Assume that $\rho(z)^2\sigma(g^2)=1$ and $\rho(g)=-1$. Then the following
	hold:
	\begin{enumerate}
		\item $X_2^{W,V}=0$ if and only if $\sigma(z)=-1$. 
		\item $X_2^{W,V}$ is simple if and only if $\sigma(z)=1$ and
			$\sigma(z)\ne-1$. In this case, $X_2^{W,V}\simeq M(gz^2,\rho_2)$,
      where $\rho_2$ is the character of $G^g$ with
			\begin{align*}
				\rho_2(g)=-\rho(z)^{-2},\quad \rho_2(z)=\rho(z).
			\end{align*}
	\end{enumerate}
\end{lem}

\begin{proof}
	Since $X_2^{W,V}=\varphi_2(W\otimes X_1^{W,V})=\K G\varphi_2(w\otimes v_1')$,
  we need
	to compute $\varphi_2(w\otimes v_1')$. Using Equation \eqref{eq:2:y1} we
	obtain that
	\begin{align*}
		\varphi_2(w\otimes v_1') &= 
		w\otimes v_1'-c_{X_1^{W,V},W}c_{W,X_1^{W,V}}(w\otimes v_1')+(\id\otimes\varphi_1)c_{1,2}(w\otimes v_1')\\
		&=(1+\sigma(z))(w\otimes v_1'-\rho(z)\sigma(z)gw\otimes v_1'),
	\end{align*}
	and hence the first claim follows. Now assume that $\sigma(z)\ne-1$ and let
	$v_1''\coloneqq(w-\rho(z)\sigma(z)gw)\otimes v_1'$. Then
	$v_1''\in (X_2^{W,V})_{gz^2}$ is non-zero. Further,
  using Lemma \ref{lem:2:Y1} we obtain that
	\begin{align*}
		gv_1'' &= (gw-\rho(z)\sigma(z)g^2w)\otimes gv_1'
		=\rho(z)^{-1}(gw-\rho(z)\sigma(g^2z)w)\otimes v_1'.
	\end{align*}
	Observe that $gv_1''=\K v_1''$ if and only if
  $\sigma(z^2)=1$, and then $gv_1''=-\sigma (g^2z)v_1''$.
  Thus $X_2^{W,V}$ is simple if and only if
	$\sigma(z)\ne-1$ and $\sigma(z)=1$. In this case, since $\sigma(z)=1$, we
	conclude that $X_2^{W,V}\simeq M(gz^2,\rho_2)$, where $\rho_2(z)=\rho(z)$ and
	$\rho_2(g)=-\rho(z)^{-2}$. 
\end{proof}

Now we define $y_n=(w-\rho(z)gw)^{\otimes n}\otimes v$ for all $n\geq0$.

\begin{lem}
	\label{lem:2:Yn}
	Assume that $\rho(z)^2\sigma(g^2)=1$, $\rho(g)=-1$, $\sigma(z)=1$, and   
	$\sigma(z)\ne-1$. Let $n\ge 1$.
  Then $\varphi _n(w\otimes y_{n-1})=ny_n$ and
  $X_n^{W,V}=n!\K Gy_n$. Moreover, if $n!\not=0$ then
  $X_n^{W,V}\simeq M(gz^n,\rho_n)$, where $\rho_n$
	is the character on $G^{g}$ given by
	\begin{align*}
		\rho_n(g)=(-1)^{n+1}\rho(z)^{-n},\quad
		\rho_n(z)=\rho(z).
	\end{align*}
\end{lem}

\begin{proof}
	It is clear that $y_n\in(W^{\otimes n}\otimes V)_{gz^n}$ for all $n\in\N$.
  Moreover,
  \[ g(w-\rho (z)gw)=gw-\rho (z)^{-1}w=-\rho (z)^{-1}(w-\rho (z)gw), \]
  and hence $gy_n=(-1)^{n+1}\rho (z)^{-n}y_n$ for all $n\in \N$.
  To prove the other claims
	we proceed by induction on $n$. For $n=1$ the claim holds by
  Lemma~\ref{lem:2:Y1}. So assume
  that the claim is valid for some $n\ge 1$. Then
  \[ X_{n+1}^{W,V}=\varphi _{n+1}(w\otimes X_n^{W,V})
     =\varphi _{n+1}(w\otimes n!y_n). \]
  Since $\sigma(z)=1$, we obtain that 
	\begin{align*}
		\varphi_{n+1}(w\otimes y_n) = w\otimes y_n-\rho (z)gw\otimes y_n
		+(w-\rho(z)gw)\otimes\varphi_n(w\otimes y_{n-1}).
	\end{align*}
  Since $\varphi _n(w\otimes y_{n-1})=ny_n$ by induction hypothesis, the
  latter equation implies that $\varphi_{n+1}(w\otimes y_n)=(n+1)y_{n+1}$.
  Therefore $X_{n+1}^{W,V}=(n+1)!\K Gy_{n+1}$.
  The remaining claims are clear.
\end{proof}

In the remaining part of this section we do not need and do not use any
assumption on the field $\K $.

\begin{pro} \label{pro:2}
	Let $V,W\in \ydG $ such that $V\simeq M(g,\rho)$,
  where $\rho$ is a character of $G^g$,
  and $W\simeq M(z,\sigma)$,
  where $\sigma$ is an absolutely irreducible representation of $G$ of degree two.
	Then $\charK\ne3$ and the following hold:
	\begin{enumerate}
		\item $(\ad V)^m(W)$ and $(\ad W)^m(V)$ are absolutely simple or zero for
      all $m\in\N_0$ if and only if
			\begin{align*}
				\rho(z)^2\sigma(g^2)=1, \quad
				\rho(g)=-1, \quad
				\sigma(z)^2=1.
			\end{align*}
		\item Assume that the conditions on $\rho $, $\sigma $ in (1) hold.
      Then 
      the Cartan matrix of $(V,W)$ satisfies
      $a_{1,2}^{(V,W)}=-2$ and
      $X_2^{V,W}\simeq M(\epsilon g^2z,\sigma_2)$, where
			$\sigma_2$ is the character of $G^\epsilon$ given by
			\[
				\sigma_2(\epsilon)=1,\quad 
				\sigma_2(z)=\rho(z)^2\sigma(z),\quad
				\sigma_2(g^2)=\sigma(g^2).
			\]
			Moreover,
			\[
			a_{2,1}^{(V,W)}=\begin{cases}
				-1 & \text{if $\sigma(z)=-1$},\\
				1-p & \text{if $\sigma(z)=1$ and $\charK=p \ge 5$},
			\end{cases}
			\]
			and $X_m^{W,V}\simeq M(gz^n,\rho_m)$, where $m=-a_{2,1}^{(V,W)}$ and 
			$\rho_m$ is the character of $G^g$ given by 
			\[
				\rho_{m}(g)=(-1)^{m+1}\rho(z)^{-m},\quad
				\rho_{m}(z)=\rho(z)\sigma(z).
			\]
	\end{enumerate}
\end{pro}

\begin{proof}
  Since $W$ is absolutely simple and $z\in Z(G)$, the representation $\sigma $
  is absolutely irreducible. Hence $\charK \not=3$ and $\sigma (1+\epsilon
  +\epsilon ^2)=0$ by Lemma~\ref{lem:simpleG3modules}.

  Since $(\ad (\mathbb{L}\otimes _\K V))^m(\mathbb{L}\otimes _\K W)\simeq
  \mathbb{L} \otimes _\K (\ad V)^m(W)$ for all $m\in \N $ and all field extensions
  $\mathbb{L}$ of $\K$, for (1) we may assume that $\K $ is algebraically
  closed. Thus (1) follows from Lemmas~\ref{lem:2:X1}, \ref{lem:2:X2},
  \ref{lem:2:X3}, \ref{lem:2:Y1}, \ref{lem:2:Y2}, and \ref{lem:2:Yn}.
  Further, under the conditions on $\rho $ and $\sigma $ in (1)
  we obtain that $a_{12}^{(V,W)}=-2$ by Lemmas~\ref{lem:2:X2} and \ref{lem:2:X3},
  and the structure of $X_2^{V,W}$ is given by Lemma~\ref{lem:2:X2}.
  Similarly, the value of $a_{21}^{(V,W)}$ and the structure of $X_m^{W,V}$
  are obtained from Lemmas~\ref{lem:2:Y2} and \ref{lem:2:Yn}.
\end{proof}

\begin{cor}
	\label{cor:R1:D}
	Assume that $\charK\ne3$.
	Let $V,W\in \ydG $ such that $V\simeq M(g,\rho)$ and
        $W\simeq M(z,\sigma)$,
        where $\rho$ is a character of $G^g$,
        and $\sigma$ is an absolutely irreducible representation of $G$ of degree two.
	Assume that
	\[
		\rho(z)^2\sigma(g^2)=1,\quad
		\rho(g)=-1,\quad 
		\sigma(z)=-1.
	\]
	Further, let $g'=g^{-1}$, $\epsilon'=\epsilon$, $z'=g^2z$,  $\rho'$ be the
	character of $G^{g'}=G^g$ dual to $\rho$, and $\sigma'$ be
	the character of $G^\epsilon $ given by $\sigma'(\epsilon)=1$,
	$\sigma'(z)=-\rho(z)^2$, $\sigma'(g^2)=\sigma(g^2)$.  Then
  $a_{1,2}^{(V,W)}=-2$ and
	\[
		R_1(V,W)=\left(V^*,X_2^{V,W}\right)
	\]
	with $V^*\simeq M(g',\rho')$, $X_2^{V,W}\simeq M(\epsilon'z',\sigma')$, and 
	\begin{align*}
		\rho'(g')=-1,\quad \sigma'(z')=-1,\quad
		\rho'(z')^2\sigma'(\epsilon'g'^2)=1,\quad
		\sigma'(\epsilon ')=1.
	\end{align*}
\end{cor}

\begin{proof}
	By Proposition~\ref{pro:2} we obtain that $a_{1,2}^{(V,W)}=-2$ and hence
	the description of $R_1(V,W)$ follows. Then
	\begin{align*}
		\sigma'(z')=\sigma'(g^2)\sigma'(z)=-\sigma(g^2)\rho(z^2)=-1.
	\end{align*}
	Similarly one proves the other formulas. 
\end{proof}

\begin{cor}
	\label{cor:R2:D}
	Assume that $\charK\ne3$.
	Let $V,W\in \ydG $ such that $V\simeq M(g,\rho)$ and
        $W\simeq M(z,\sigma)$,
        where $\rho$ is a character of $G^g$,
        and $\sigma$ is an absolutely irreducible representation of $G$ of degree two.
	Assume that
	\[
	\rho(z)^2\sigma(g^2)=1,\quad 
	\rho(g)=-1,\quad
	\sigma(z)=-1. 
	\]
	Let $g''=gz$, $\epsilon''=\epsilon^{-1}$, $z''= z^{-1}$, $\rho''$ be
	the character of $G^g$ given by $\rho''(g)=\rho(z)^{-1}$
	and $\rho''(z)=-\rho(z)$, and
	$\sigma''$ be the degree two representation of $G$ dual to
  $\sigma$. Then $a_{2,1}^{(V,W)}=-1$ and
	\[
		R_2(V,W)=\left(X_1^{W,V},W^*\right)
	\]
	with $X_1^{W,V}\simeq M(g'',\rho'')$, $W^*\simeq M(z'',\sigma '')$,
  and 
	\begin{align*}
    \rho''(z'')^2\sigma''(g''^2)=1,\quad
    \rho''(g'')=-1,\quad
		\sigma''(z'')=-1.
	\end{align*}
\end{cor}

\begin{proof}
	It is similar to the proof of Corollary~\ref{cor:R1:D}.
\end{proof}

\section{Reflections of the third pair}
\label{section:3}

Let $V,W\in \ydG $ such that $V\simeq M(g,\rho )$,
where $\rho $ is a character of $G^g$, and
$W\simeq M(z,\sigma)$, where $\sigma$ is a character of $G$.
Let $w\in W=W_{z}$ with $w\ne0$. Then $\{w\}$ is a
basis of $W$. Since $g\epsilon=\epsilon^{-1}g$ and $\epsilon^3=1$, we obtain
that
\[
	gw=\sigma(g)w,\quad
	\epsilon w=w,\quad
	zw=\sigma(z)w.
\]

We first compute the modules $(\ad V)^m(W)$ for $m\in\N$. 

\begin{lem}
	\label{lem:3:X1}
	The Yetter-Drinfeld module $X_1^{V,W}$ is non-zero if and only if
  $\rho(z)\sigma(g)\not=1$. In this case, $X_1^{V,W}$ is
	absolutely simple and $X_1^{V,W}\simeq M(gz,\sigma_1)$,
  where $\sigma_1$ is the
	character of $G^{g}$ given by 
	\[
		\sigma_1(g)=\rho(g)\sigma(g),\quad
		\sigma_1(z)=\rho(z)\sigma(z).
	\]
	Let $w'=v\otimes w$. Then $\{w',\epsilon w', \epsilon^2w'\}$ is a basis
	of $X_1^{V,W}$. The degrees of these basis vectors are $gz$, $g\epsilon z$, and
	$g\epsilon ^2z$, respectively.
\end{lem}

Again, the action of $G$ on $X_1^{V,W}$ can be displayed in a table similar to the
one in Remark~\ref{rem:1:G_on_V}, where $v$ has to be replaced by $w'$ and
$\rho $ has to be replaced by $\sigma _1$.

\begin{proof}
	Write $X_1^{V,W}=\varphi_1(V\otimes W)=\K G\varphi_1(v\otimes w)$ and compute
	\[
		\varphi_1(v\otimes w)=v\otimes w-c_{W,V}c_{V,W}(v\otimes w)
    =(1-\rho (z)\sigma(g))v\otimes w.
	\]
	Then $w'=v\otimes w\in(V\otimes W)_{gz}$ is non-zero and $X_1^{V,W}=0$ if
  and only if $\rho(z)\sigma(g)=1$.
  Assume that $\rho(z)\sigma(g)\ne1$. Then
	\[
		gw'=\rho(g)\sigma(g)w',\quad
		zw'=\rho(z)\sigma(z)w'.
	\]
	Hence $X_1^{V,W}=\K Gw'\simeq M(gz,\sigma_1)$, and the lemma follows.
\end{proof}

\begin{lem}
	\label{lem:3:X2}
	Assume that $\rho(z)\sigma(g)\ne1$. Then $X_2^{V,W}\neq 0$.
  Moreover, $X_2^{V,W}$ is absolutely simple if and
	only if
  \[ \rho (g)=-1,\quad (3)_{-\rho(z)\sigma(g)}=0 \quad \text{or}\quad
    \rho (gz)\sigma (g)=1,\quad (3)_{-\rho (g)}=0.
  \]
  In both cases, 
	$X_2^{V,W}\simeq M(\epsilon g^2z,\sigma_2)$, where $\sigma_2$ is the
  character of $G^\epsilon$ with
	\begin{align*}
		\sigma_2(\epsilon)=\rho(g)(1-\rho(z)\sigma(g)),\quad
		\sigma_2(g^2)=\rho(g^4)\sigma(g^2),\quad
		\sigma_2(z)=\rho(z^2)\sigma(z).
	\end{align*}
	Let $w''=\varphi_2(\epsilon^2v\otimes w')$. Then
	a basis of $X_2^{V,W}$ is given by $\{w'',gw''\}$. The degrees of these basis
	vectors are $\epsilon g^2z$ and $\epsilon^2g^2z$ respectively.
\end{lem}

\begin{proof}
	Write $X_2^{V,W}=\K G\varphi_2(v\otimes w')\oplus
  \K G\varphi_2(\epsilon^2v\otimes w')$. Then compute
	\begin{align*}
		\varphi_2(v\otimes w') &=
                (\id-c_{X_1^{V,W},V}c_{V,X_1^{V,W}})(v\otimes w')
                +(\id\otimes\varphi_1)c_{1,2}(v\otimes v\otimes w)\\
		&= (1+\rho(g))(1-\rho(gz)\sigma(g))v\otimes w',
	\end{align*}
	and, using $\varphi_1(\epsilon^2v\otimes w)=(1-\rho(z)\sigma(g))\epsilon^2w'$, 
	\begin{equation*}
		\begin{aligned}
			w'' &= \varphi_2(\epsilon^2v\otimes w')\\
			&= (\id-c_{X_1^{V,W},V}c_{V,X_1^{V,W}})(\epsilon^2v\otimes w')
      +(\id\otimes\varphi_1)c_{1,2}(\epsilon ^2v\otimes v\otimes w)\\
			&= \epsilon^2v\otimes w'-\rho(g^2z)\sigma(g)v\otimes\epsilon w'
      +\rho(g)(1-\rho(z)\sigma(g))\epsilon v\otimes\epsilon^2w'.
		\end{aligned}
	\end{equation*}
	Hence $w''\in (V\otimes X_1^{V,W})_{\epsilon g^2z}$ is non-zero.
	Therefore $w''$ is absolutely simple
	if and only if $(1+\rho(g))(1-\rho(gz)\sigma(g))=0$ and $\epsilon
	w''=\K w''$. Since
	\[
		\epsilon w''=v\otimes\epsilon w'
    -\rho(g^2z)\sigma(g)\epsilon v\otimes\epsilon^2w'+\rho(g)(1-\rho(z)\sigma(g))\epsilon^2v\otimes w',
	\]
	in the case $\rho (g)=-1$ we have $\epsilon w''=\K w''$ if and only if
  $(3)_{-\rho (z)\sigma (g)}=0$,
  and then $\epsilon w''=(\rho (z)\sigma (g)-1)w''$.
  Similarly, if $\rho (gz)\sigma (g)=1$, then
	$\epsilon w''=\K w''$ if and only if
  $(3)_{-\rho (g)}=0$,
  and then $\epsilon w''=\rho (g)(1-\rho (z)\sigma (g))w''$.
	The rest is clear.
\end{proof}

\begin{lem} \label{lem:3:X3}
	Assume that $\rho(z)\sigma(g)\ne1$ and that $X_2^{V,W}$ is absolutely simple.
        Then $X_3^{V,W}\ne0$ if and only if $\rho(g)\ne-1$. In this case, $\charK \not=3$,
	$\rho(gz)\sigma(g)=1$, $(3)_{-\rho (g)}=0$, the Yetter-Drinfeld module $X_3^{V,W}$
        is absolutely simple, and $X_3^{V,W}\simeq M(g^3z,\sigma_3)$, where
	\begin{align*}
		\sigma_3(g)=\sigma(g), \quad \sigma_3(z)=\rho(z)^3\sigma(z).
	\end{align*}
\end{lem}

\begin{proof}
	Acting with $\epsilon g$ on $w''=\varphi_2(\epsilon^2v\otimes w')$ we obtain that
	\begin{equation*}
		\varphi_2(\epsilon^2v\otimes\epsilon w')=\sigma(g)^{-1}(1-\rho(z)\sigma(g))^2gw''.
	\end{equation*}
	Further,
	$\varphi_2(\epsilon^2v\otimes\epsilon^2w')=\epsilon^2\varphi_2(v\otimes
	w')=0$.  A direct calculation using these formulas and the expression for
	$w''$ yields	
	\begin{multline*}
		\varphi_3(\epsilon^2v\otimes w'')
		=(1+\rho(g))(\epsilon^2v\otimes w''-\rho(g^3z)(1-\rho(z)\sigma(g))^2\epsilon v\otimes gw'').
	\end{multline*}
	Thus $X_3^{V,W}=0$ if and only if $\rho (g)=-1$.
        Assume now that $\rho(g)\ne-1$. Since $X_2^{V,W}$ is absolutely simple,
	Lemma~\ref{lem:3:X2} implies that
	$\rho(gz)\sigma(g)=1$ and $(3)_{-\rho(g)}=0$. Then $\charK \not=3$, since otherwise
        $(1+\rho (g))^2=0$, a contradiction to $\rho (g)\not=-1$.
        Let $w'''=(1+\rho (g))^{-1}\varphi _3(\epsilon ^2v\otimes w'')$. Then
	\begin{align*}
		w'''=\epsilon^2v\otimes w''+\rho(g^2z)\epsilon v\otimes gw'',
	\end{align*}
	and hence $w'''\in(V\otimes X_2^{V,W})_{g^3z}$ is non-zero, $X_3^{V,W}=\K Gw'''$, 
	$gw'''=\sigma(g)w'''$, and $zw'''=\rho(z)^3\sigma(z)w'''$.
	Therefore $X_3^{V,W}\simeq M(g^3z,\sigma_3)$, where $\sigma _3$ is the
        character of $G^g$ with $\sigma_3(g)=\sigma(g)$ and
	$\sigma_3(z)=\rho(z)^3\sigma(z)$.
\end{proof}

\begin{lem}
	\label{lem:3:X4}
	Assume that $\charK \not=3$, $\rho (gz)\sigma (g)=1$, and $(3)_{-\rho (g)}=0$.
        Then $X_4^{V,W}\not=0$, and $X_4^{V,W}$ is absolutely simple if
	and only if $\charK=2$. In this case, $X_4^{V,W}\simeq M(g^4z,\sigma_4)$,
	where $\sigma_4$ is the character of $G$ given by
	\begin{align*}
		\sigma_4(g)=\rho(g)\sigma(g),\quad
		\sigma_4(\epsilon)=1,\quad
		\sigma_4(z)=\rho(z)^4\sigma(z),
	\end{align*}
	and $X_5^{V,W}=0$. 
\end{lem}

\begin{proof}
	The assumptions imply that $\rho (z)\sigma (g)=\rho (g)^{-1}\not=1$,
        and $\rho (g)\not=-1$ since $\charK \not=3$. Therefore
        $X_n^{V,W}$ is absolutely simple for all $n\in \{1,2,3\}$ by Lemmas~\ref{lem:3:X1},
        \ref{lem:3:X2}, and \ref{lem:3:X3}.
	By looking at the support of $V\otimes X_3^{V,W}$ we also know that
        $X_4^{V,W}=\K G\varphi_4(\epsilon^2v\otimes w''')\oplus \K G\varphi_4(v\otimes w''')$.
	
	We first obtain that
	\begin{multline*}
	  \varphi_4(v\otimes w''') = v\otimes w'''-c_{X_3,V}\,c_{V,X_3}(v\otimes w''')
          +\rho(g)\epsilon v\otimes\varphi_3(v\otimes w'')\\
	  +\rho(g^3z)\epsilon^2v\otimes\varphi_3(v\otimes gw'').
	\end{multline*}
	Now act with $\epsilon$ and $g\epsilon$ on 
	\begin{equation}
		\label{eq:phi3(e2v,x2)}
		\varphi_3(\epsilon^2v\otimes w'')=(1+\rho(g))w'''
	\end{equation}
	to obtain that
	\begin{align*}
		&\varphi_3(v\otimes w'')=\rho(g)^{-2}(1+\rho(g))\epsilon w''',\\
		&\varphi_3(v\otimes gw'')=\rho(g)^{-3}(1+\rho(g))\sigma(g)\epsilon^2w''',
	\end{align*}
	respectively. Then
	\begin{align} \label{eq:varphi4:3}
		\varphi_4(v\otimes w''')=(1+\rho(g)^{-1})(v\otimes w'''+\epsilon v\otimes\epsilon w'''
                +\epsilon^2v\otimes\epsilon^2w''')
	\end{align}
	and hence $\varphi_4(v\otimes w''')\in(V\otimes X_3^{V,W})_{g^4z}$ is non-zero.

	Act with $\epsilon g$ on Equation \eqref{eq:phi3(e2v,x2)} and use that
	 $\sigma_2(\epsilon)=\rho(g^2)$. Thus
	\begin{align*}
		\varphi_3(\epsilon^2v\otimes gw'')=(1+\rho(g))\rho (g)\sigma(g)\epsilon w'''.
	\end{align*}
	Now compute
	\begin{align*}
		\varphi_4(\epsilon^2v\otimes w''')=&\;\epsilon^2v\otimes w'''
                -\rho(g^2)v\otimes \epsilon w'''\\
                &\;+\rho(g)\epsilon^2v\otimes\varphi_3(\epsilon^2v\otimes w'')
		-\rho(z)v\otimes\varphi_3(\epsilon^2v\otimes gw'')\\
		=&\;(1+\rho(g)+\rho(g)^2)(\epsilon^2v\otimes w'''-v\otimes \epsilon w''').
	\end{align*}
        Thus $X_4^{V,W}$ is absolutely simple if and only if $(3)_{\rho (g)}=0$ (in which case $\charK=2$)
        and $\K G\varphi _4(v\otimes w''')$ is absolutely simple.
	Let
        $$w''''=v\otimes w'''+\epsilon v\otimes \epsilon w'''+\epsilon^2v\otimes\epsilon^2w'''.$$
        Then
	$\varphi_4(v\otimes w''')=(1+\rho(g)^{-1})w''''$
        by Equation~\eqref{eq:varphi4:3}, $gw''''=\rho (g)\sigma (g)w''''$, and hence
        $X_4^{V,W}\simeq M(g^4z,\sigma_4)$, where $\sigma_4$ is the
	character of $G$ given by $\sigma_4(\epsilon)=1$,
	$\sigma_4(g)=\rho(g)\sigma(g)$, and $\sigma_4(z)=\rho(z)^4\sigma(z)$. 

	Now we prove that $X_5^{V,W}=0$. Observe that
        $$X_5^{V,W}=\varphi_5(V\otimes X_4^{V,W}) =\K G\varphi_5(v\otimes w'''').$$
        By acting with $\epsilon$ and $g\epsilon $ on
	$\varphi_4(\epsilon^2v\otimes w''')=0$ we obtain that $\varphi_4(v\otimes\epsilon
	w''')=0$ and $\varphi_4(v\otimes\epsilon^2w''')=0$, respectively. A direct calculation yields
        that
	\begin{align*}
		\varphi_5(v\otimes w'''')=(1+\rho(g))v\otimes w''''
                +\rho(g)v\otimes\varphi_4(v\otimes w''')=0,
	\end{align*}
	which proves the claim.
\end{proof}

Now we compute the modules $(\ad W)^m(V)$ for $m\geq1$. As before, we write
$\varphi_n=\varphi_n^{W,V}$. 

\begin{lem} \label{lem:3:Yn}
	Let $y_0=v$ and let $y_n=w\otimes y_{n-1}$ for all $n\geq1$.
        Then $y_n\in(W^{\otimes n}\otimes V)_{gz^n}$ and
        $\K Gy_n\simeq M(gz^n,\rho_n)$,
        where $\rho_n$ is the character of $G^g$ given by
	\begin{align*}
		\rho_n(g)=\rho(g)\sigma(g)^n, \quad
		\rho_n(z)=\rho(z)\sigma(z)^n,
	\end{align*}
	and 
	$\varphi_n(w\otimes y_{n-1})=\gamma_n y_n$ for all $n\in\N$, where
	\[
		\gamma_n=(n)_{\sigma(z)}(1-\rho (z)\sigma(gz^{n-1})).
	\]
	Moreover, 
	\[
		X_n^{W,V}\simeq\K G((\gamma_1\cdots\gamma_n)y_n)
	\]
	for all $n\in\N_0$. 
\end{lem}

\begin{proof}
    We prove by induction on $n$ that $\varphi_n(w\otimes y_{n-1})=\gamma_n y_n$
    for all $n\ge 1$. The remaining claims are then easily shown.

    It is clear that $\varphi _1(w\otimes v)=w\otimes v-gw\otimes zv 
    =(1-\rho (z)\sigma (g))y_1$.
    Let now $n\ge 1$. Then
    \begin{align*}
      \varphi_{n+1}(w\otimes y_n)=w\otimes y_n-gz^nw\otimes zy_n
      +\sigma(z)w\otimes\varphi_n(w\otimes y_{n-1}),
    \end{align*}
    and hence induction hypothesis implies that
	\begin{align*}
		\varphi_{n+1}(w\otimes y_n)=(n+1)_{\sigma(z)}(1-\rho (z)\sigma(gz^n))y_{n+1}.
	\end{align*}
    This proves the claimed formula.
\end{proof}

\begin{pro} \label{pro:3}
	Let $V,W\in \ydG $ such that $V\simeq M(g,\rho)$,
  where $\rho$ is a character of $G^g$,
  and $W\simeq M(z,\sigma)$,
  where $\sigma$ is a character of $G$.
	Then the following hold:
	\begin{enumerate}
		\item $(\ad V)^m(W)$ and $(\ad W)^m(V)$ are absolutely simple or zero for
			all $m\in\N_0$ if and only if 
			\begin{enumerate}
				\item $\rho(z)\sigma(g)=1$, or
				\item $\rho(g)=-1$ and $(3)_{-\rho(z)\sigma(g)}=0$, or
				\item $\rho(gz)\sigma(g)=1$, $(3)_{\rho(g)}=0$, and $\charK=2$.
			\end{enumerate}
		\item Assume that the equations in one of (1a), (1b), (1c) hold. Then 
        the Cartan matrix of the pair $(V,W)$ satisfies
			\begin{align*}
				a_{1,2}^{(V,W)}=\begin{cases}
					0 & \text{in the case (1a)},\\
					-2 & \text{in the case (1b)},\\
					-4 & \text{in the case (1c)}.
				\end{cases}
			\end{align*}
			Moreover, 
			\begin{enumerate}
				\item if $a_{1,2}^{(V,W)}=-2$ then $X_2^{V,W}\simeq M(\epsilon
					g^2z,\sigma_2)$, where $\sigma_2$ is the character
                                        of $G^\epsilon$ given by
					$\sigma_2(\epsilon)=-\rho(z)^{-1}\sigma(g)^{-1}$,
					$\sigma_2(z)=\rho(z)^2\sigma(z)$ and $\sigma_2(g^2)=\sigma(g^2)$, 
				\item if $a_{1,2}^{(V,W)}=-4$ then $X_4^{V,W}\simeq M(g^4z,\sigma_4)$,
					where $\sigma_4$ is the character of $G$ given by
					$\sigma_4(g)=\rho(g)\sigma(g)$,
					$\sigma_4(\epsilon)=1$, and $\sigma_4(z)=\rho(z)^4\sigma(z)$.
			\end{enumerate}
		\item Assume that the equations in one of (1a), (1b), (1c) hold and let
      $m\in\N_0$.
      Then $a_{2,1}^{(V,W)}=-m$ if
			and only if $\gamma_{m+1}=0$ and $\gamma_1\cdots\gamma_m\ne0$, where
			\[
			\gamma_k=(k)_{\sigma(z)}(1-\rho (z)\sigma(gz^{k-1}))
			\]
			for all $k\in\N_0$. In this case,
			$X_m^{W,V}\simeq M(gz^m,\rho_m)$, where $\rho_m$ is the
			character of $G^g$ given by 
			\[
				\rho_m(g)=\rho (g)\sigma(g)^m,\quad 
				\rho_m(z)=\rho (z)\sigma(z)^m.
			\]
	\end{enumerate}
\end{pro}

\begin{proof}
  By Lemmas~\ref{lem:3:X1}, \ref{lem:3:X2}, \ref{lem:3:X3}, and
  \ref{lem:3:X4}, $(\ad V)^m(W)$ is absolutely simple or zero for all $m\in \N $
  if and only if $\rho (z)\sigma (g)=1$ or $\rho (g)=-1$, $(3)_{-\rho
  (z)\sigma (g)}=0$, or $\rho (gz)\sigma (g)=1$, $(3)_{-\rho (g)}=0$, $\charK
  \not=3$, $\charK=2$. 
  By Lemma~\ref{lem:3:Yn} the Yetter-Drinfeld modules $(\ad W)^m(V)$ for $m\ge
  0$ are absolutely simple or zero. This proves (1). Then (2) is easy to get
  from the same lemmas.
\end{proof}

\begin{cor}
	\label{cor:R1:E}
	Let $V,W\in \ydG $ such that $V\simeq M(g,\rho)$,
  where $\rho$ is a character of $G^g$,
  and $W\simeq M(z,\sigma)$,
  where $\sigma$ is a character of $G$, and
	\[
		\rho(g)=-1,\quad
		\rho(z)\sigma(gz)=1,\quad
		1-\rho(z)\sigma(g)+\rho(z^2)\sigma(g^2)=0.
	\]
	Further, let $g'=g^{-1}$, $\epsilon'=\epsilon$, $z'= g^2z$, $\rho'$ be the
	character of $G^g$ dual to $\rho$, and $\sigma'$ the
	character of $G^\epsilon $ given by
	$\sigma'(\epsilon)=-\rho(z)^{-1}\sigma(g)^{-1}$,
	$\sigma'(z)=\rho(z)^2\sigma(z)$, $\sigma'(g^2)=\sigma(g^2)$.  Then 
  $a_{1,2}^{(V,W)}=-2$ and
	\[
		R_1(V,W)=\left(V^*,X_2^{V,W}\right),
	\]
	with $V^*\simeq M(g',\rho')$, $X_2^{V,W}\simeq M(\epsilon'z',\sigma')$, and
	\begin{align*}
		\rho'(g')=\sigma'(\epsilon'z')=-1,\quad
    \rho'(z')^2\sigma'(\epsilon'g'^2)=1,\quad
		1+\sigma'(\epsilon')+\sigma'(\epsilon')^2=0.
	\end{align*}
\end{cor}

\begin{proof}
	Using Proposition \ref{pro:3} one obtains that $a_{1,2}^{(V,W)}=-2$. The rest
	of the proof is similar to the proof of Corollary \ref{cor:R1:A}.
\end{proof}

\begin{cor} \label{cor:R2:E}
	Let $V,W\in \ydG $ such that $V\simeq M(g,\rho)$,
  where $\rho$ is a character of $G^g$,
  and $W\simeq M(z,\sigma)$,
  where $\sigma$ is a character of $G$, and
	\[
		\rho(g)=-1,\quad
		\rho(z)\sigma(gz)=1,\quad
		1-\rho(z)\sigma(g)+\rho(z^2)\sigma(g^2)=0.
	\]
	Let $g''=gz$. $\epsilon''=\epsilon^{-1}$, $z''=z^{-1}$,  $\rho''$ be the
	character of $G^g$ given by $\rho''(g)=-\sigma(g)$,
	$\rho''(z)=\rho(z)\sigma(z)$, and $\sigma''$
  be the character of $G$ dual to $\sigma$. Then $a_{2,1}^{(V,W)}=-1$ and
	\[
		R_2(V,W)=\left(X_1^{W,V},W^*\right),
	\]
	with 
	$X_1^{W,V}\simeq M(g'',\rho'')$,
	$W^*\simeq M(z'',\sigma'')$, and 	
	\begin{align*}
	  \rho''(g'')=-1,\quad \rho''(z'')\sigma''(g''z'')=1,\quad
          (3)_{-\rho''(z'')\sigma''(g'')}=0.
	\end{align*}
\end{cor}

\begin{proof}
	Since $\rho (z)\sigma (g)\not=1$,
  Proposition \ref{pro:3} implies that $a_{2,1}^{(V,W)}=-1$. The rest
	of the proof is similar to the proof of Corollary \ref{cor:R2:A}.
\end{proof}

\section{Computing the reflections}
\label{section:reflections}

\subsection{Pairs of Yetter-Drinfeld modules}

Let $G$ be a non-abelian epimorphic image of $\Gamma_3$. We first identify pairs
$(V,W)$ of Yetter-Drinfeld modules over $G$ with a Cartan matrix of
finite type. More precisely, we define classes $\wp _i$, where $1\le i\le 6$,
of pairs of absolutely simple Yetter-Drinfeld modules over $G$, such that
all pairs in these classes can be treated simultaneously with respect to
reflections.

\begin{defin} (\emph{The classes $\wp_i$ for $1\le i\le 6$ of pairs of
  Yetter-Drinfeld modules}) \label{def:wp}
  Let $V$ and $W$ be Yetter-Drinfeld modules over $G$ and let $i\in \N $ with
  $1\le i\le 6$. We say that
  $(V,W)\in \wp _i$, if there exist $g,\epsilon ,z\in G$
  such that the following hold.
  \begin{enumerate}
    \item There is a group epimorphism
    \[ \Gamma _3\to G,\qquad \gamma \mapsto g,\quad \nu \mapsto \epsilon ,\quad
       \zeta \mapsto z.
    \]
    \item $V\simeq M(g,\rho )$ for some character $\rho $ of $G$, and
      either $W\simeq M(\epsilon z,\sigma )$ for some character $\sigma $ of
      $G^\epsilon $ or $W\simeq M(z,\sigma )$ for some absolutely irreducible
      representation $\sigma $ of $G$.
    \item $W$, $\rho $, $\sigma $, and $\K $ satisfy the conditions in
      the $i$-th row of Table~\ref{tab:finitetype}.
  \end{enumerate}
\end{defin}

	\begin{table}[h]
    \caption{Conditions on the classes $\wp_i$, $1\le i\le 6$}
\begin{center}
		\begin{tabular}{|c|c|c|c|}
			\hline 
			& $[W]$ & Conditions on $\rho$ and $\sigma$ & $\charK$ \tabularnewline
			\hline
			$\wp_1$ & $M(\epsilon z,\sigma)$ & $\rho(g)=\sigma(\epsilon z)=-1$ & \tabularnewline
			&&$\rho(z^{2})\sigma(\epsilon g^{2})=1$, $(3)_{\sigma(\epsilon)}=0$ & \tabularnewline
			\hline 
			$\wp_2$ & $M(\epsilon z,\sigma)$ & $\rho(g)=\sigma(z)=-1$& $\ne3$\tabularnewline
			&&$\rho(z^{2})\sigma(\epsilon g^{2})=\sigma(\text{\ensuremath{\epsilon})=1}$ & \tabularnewline 
			\hline 
			$\wp_3$ & $M(z,\sigma)$ & $\rho(g)=\sigma(z)=-1$ & $\ne3$\tabularnewline
			&& $\rho(z^{2})\sigma(g^{2})=1$, $\deg\sigma=2$ & \tabularnewline
			\hline 
			$\wp_4$ & $M(z,\sigma)$ & $\rho(g)=-1$, $\rho(z)\sigma(gz)=1$ & \tabularnewline
			&& $(3)_{-\rho(z)\sigma(g)}=0$, $\deg\sigma=1$ &  \tabularnewline
			\hline 
			$\wp_5$ & $M(z,\sigma)$ & $\rho(g)=\sigma(z)=1$,
        $(3)_{\rho(z)\sigma(g)}=0$& $2$ \tabularnewline
			&& $\deg\sigma=1$ & \tabularnewline
			\hline 
			$\wp_6$ & $M(z,\sigma)$ & $\rho(g)=\sigma(z)=-1$, $(3)_{-\rho(z)\sigma(g)}=0$ & $\ne2,3$ \tabularnewline
			&& $\deg\sigma=1$ & \tabularnewline
			\hline 
		\end{tabular}
		\label{tab:finitetype}
\end{center}
	\end{table}

\begin{pro}
	\label{pro:pairs}
  Let $V,W$ be absolutely simple Yetter-Drinfeld modules over $G$.
  Suppose that $c_{W,V}c_{V,W}\ne\id_{V\otimes W}$, the pair
	$(V,W)$ admits all reflections, the Weyl groupoid of $(V,W)$ is finite,
  and $|\supp V|\ge |\supp W|$.
	Then the Cartan matrix of $(V,W)$ is of finite type if and only if
	$(V,W)$ belongs to one of the classes $\wp _i$ for $1\le i\le 5$
  in Table \ref{tab:finitetype}. In this case
  $a_{1,2}^{(V,W)}=-2$ and $a_{2,1}^{(V,W)}=-1$.
\end{pro}

\begin{proof}
  {}From \cite[Thm.\,7.3]{examples} we know that the quandle
  $\supp(V\oplus W)$ is isomorphic to $Z_3^{3,1}$ or $Z_3^{3,2}$.
  Since $|\supp V|\ge |\supp W|$, we conclude that
  $|\supp V|=3$ and $|\supp W|\le 2$. Let $g,g'\in \supp V$ with $g'\neq g$
  and let $\epsilon =g'{}^{-1}g$. Then $\epsilon \not=1$ and $\epsilon ^3=1$
  by \eqref{eq:classes}. Since $V$ is absolutely simple,
  there exists a character $\rho $ of $G$ such that $V\simeq M(g,\rho )$.
  Since $W$ is absolutely simple and $\supp (V\oplus W)$ generates $G$,
  there exists $z\in Z(G)$ such that
  either $W\simeq M(\epsilon z,\sigma )$ for a character $\sigma $ of
  $G^\epsilon $, or $W\simeq M(z,\sigma )$ for some absolutely irreducible
  representation $\sigma $ of $G$. Then $\deg \sigma \le 2$ by
  Lemma~\ref{lem:simpleG3modules}.
 
	Since the pair $(V,W)$ admits all reflections and the Weyl groupoid of
  $(V,W)$ is finite, we conclude from \cite[Theorem\,2.5]{MR2732989}
  that $(\ad V)^m(W)$ and $(\ad W)^m(V)$ are
  absolutely simple or zero for all $m\ge 1$. The condition
  $c_{W,V}c_{V,W}\not=\id _{V\otimes W}$
  just means that the Cartan matrix of $(V,W)$
  is not of type $A_1\times A_1$.
  Propositions~\ref{pro:1}, \ref{pro:2}, and \ref{pro:3} imply that
	the Cartan matrix of $(V,W)$ is of finite type if and only if
	$(V,W)$ belongs to one of the classes $\wp _i$ for $1\le i\le 6$
  in Table \ref{tab:finitetype}. Moreover, in this case $a_{1,2}^{(V,W)}=-2$
  and $a_{2,1}^{(V,W)}=-1$.

  We have to exclude the class $\wp _6$. To do so, it suffices to prove that
  if $(V,W)\in \wp _6$ then not all assumptions of the proposition are
  fulfilled.

  Assume that $(V,W)$ belongs to $\wp _6$. By Proposition~\ref{pro:3},
  \[ R_2(V,W)=( (\ad W)(V),W^*)
  \simeq (M(gz,\rho _1),M(z^{-1},\sigma ^{-1})),\]
  where $\rho _1$ is the character of $G^g$ with
  $\rho _1(g)=\rho (g)\sigma (g)$ and $\rho _1(z)=\rho (z)\sigma (z)$.
  Let $(V',W')=R_2(V,W)$.
  We obtain that $\rho _1(gz)=\rho (z)\sigma (g)\notin \{1,-1\}$,
  since $(3)_{-\rho (z)\sigma (g)}=0$ and $\charK \not=3$. Since also $\charK
  \not=2$ by assumption on $\wp _6$,
  Proposition~\ref{pro:3} implies that
  not all of $(\ad V')^m(W')$ and $(\ad W')^m(V')$ for $m\ge 1$
  are absolutely simple or
  zero. This is a
  contradiction to the assumption that $(V,W)$ admits all reflections and the
  Weyl groupoid of $(V,W)$ is finite.
\end{proof}

Let $(V,W)$ be a pair of Yetter-Drinfeld modules over $G$.  
If $(V,W)\in \wp _1$, then $R_1(V,W)\in \wp _4$ and $R_2(V,W)\in \wp _1$ by
Corollaries~\ref{cor:R1:A} and \ref{cor:R2:A}.  On the other hand, if $(V,W)\in
\wp _4$, then $R_1(V,W)\in \wp _1$ and $R_2(V,W)\in \wp _4$ by
Corollaries~\ref{cor:R1:E} and \ref{cor:R2:E}.  We display this fact in the
following graph: \begin{equation} \xymatrix{{\wp_1}\ar@{-}[r]^{R_1} & \wp_4}
\end{equation} We omit $R_2$ in the graph, since it fixes the classes
$\wp _1$ and $\wp _4$.  Since
the Cartan matrices of $(V,W)$ are the same for all reflections of $(V,W)$ in
the classes $\wp _1$ and $\wp _4$, we conclude the following.

\begin{lem} \label{lem:wp1wp4standard}
  Let $(V,W)$ be a pair in $\wp _1$ or $\wp _4$. Then $(V,W)$
  admits all reflections and the Weyl groupoid of $(V,W)$
  is standard of type $B_2$.
\end{lem}

Similarly to the previous paragraph, Corollaries~\ref{cor:R1:B},
\ref{cor:R2:B}, \ref{cor:R1:D}, and \ref{cor:R2:D} imply that the reflections
of the pairs $(V,W)$ in $\wp _2$ and $\wp _3$ can be displayed in the following
graph:
\begin{equation}
	\xymatrix{\wp_2\ar@{-}[r]^{R_1} & \wp_3}     
\end{equation}

\begin{lem} \label{lem:wp2wp3standard}
  Let $(V,W)$ be a pair in $\wp _2$ or $\wp _3$. Then $(V,W)$
  admits all reflections and the Weyl groupoid of $(V,W)$
  is standard of type $B_2$.
\end{lem}

The reflections of the pairs $(V,W)\in \wp _5$ show a more complicated pattern.

Let us assume that $\charK=2$.
Lemmas~\ref{lem:R1:F}, \ref{lem:R2:F}, \ref{lem:R1:F1}, \ref{lem:R2:F1},
\ref{lem:R1:F2}, and \ref{lem:R2:F2} below
imply that the reflections of the pairs
$(V,W)$ in $\wp _5$ can be displayed in the following graph:
\begin{equation}
\xymatrix{\wp_5'\ar@{-}[r]^{R_1} & \wp_5\ar@{-}[r]^{R_2} & \wp_5''}     
\end{equation}
The classes $\wp_5'$ and $\wp_5''$ are defined in the same way as $\wp _5$
in Definition~\ref{def:wp}, except that in the last line one refers to the
conditions in Table~\ref{tab:pair5}. 
Since 
the class $\wp _5$ is non-empty only if
$\charK=2$, in the definitions of $\wp_5'$ and $\wp_5''$ we also assume that
$\charK =2$.

    \begin{table}[h]
        \caption{Conditions on the classes $\wp_5$, $\wp _5'$, $\wp _5''$}
\begin{center}
        \begin{tabular}{|r|c|c|}
            \hline 
            & $[W]$ &
            \rule{0pt}{2.3ex}
            Conditions on $\rho$ and $\sigma$ \tabularnewline
            \hline
            $\wp_5$ & $M(z,\sigma)$ &
            \rule{0pt}{2.3ex}
            $\rho(g)=\sigma(z)=1$,
            $(3)_{\rho(z)\sigma(g)}=0$, $\deg\sigma=1$ \tabularnewline
            \hline 
            $\wp_5'$ & $M(\epsilon z,\sigma)$ & 
            \rule{0pt}{2.3ex}
            $(3)_{\sigma(\epsilon)}=0$, $\sigma(z)=\sigma(\epsilon)$,
            $\rho(z^2)\sigma(\epsilon g^2)=1$, $\rho(g)=1$ \tabularnewline
            \hline
            \rule{0pt}{2.3ex}
            $\wp_5''$ & $M(z,\sigma)$ &  
            $(3)_{\rho(g)}=0$, $\sigma(z)=1$,
            $\rho(gz)\sigma(g)=1$, $\deg\sigma=1$
            \tabularnewline 
            \hline
        \end{tabular}
        \label{tab:pair5}
\end{center}
    \end{table}

\begin{lem}
	\label{lem:R1:F}
	Let $(V,W)\in \wp _5$ and $g,\epsilon ,z,\rho
  ,\sigma $ as in Definition~\ref{def:wp} such that
  $V\simeq M(g,\rho)$ and $W\simeq M(z,\sigma)$.
	Let $g'=g^{-1}$, $\epsilon'=\epsilon$, $z'=g^2z$, $\rho'$ be the character
  of $G^g$ dual to $\rho$, and $\sigma'$ be the
	character of $G^\epsilon $ given by
  $\sigma'(\epsilon)=-\rho(z)^{-1}\sigma(g)^{-1}$,
  $\sigma'(z)=\rho(z)^2$, $\sigma'(g^2)=\sigma(g^2)$. Then
  $a_{1,2}^{(V,W)}=-2$ and
	\[
		R_1(V,W)=\left(V^*,X_2^{V,W}\right)
	\]
	with $V^*\simeq M(g',\rho')$, $X_2^{V,W}\simeq M(\epsilon'z',\sigma')$, and 
	\begin{align*}
		\rho'(g')=1,\quad \rho'(z'^2)\sigma'(\epsilon'g'^2)=1,\quad
		\sigma'(z')=\sigma'(\epsilon'),\quad
		(3)_{\sigma'(\epsilon')}=0.
	\end{align*}
  In particular, $R_1(V,W)\in \wp _5'$.
\end{lem}

\begin{proof}
	Using Proposition \ref{pro:3} we obtain that $a_{1,2}^{(V,W)}=-2$ and hence
	the description of $R_1(V,W)$ together with the isomorphisms regarding $V^*$
  and $X_2^{V,W}$ follow from the same Proposition. The remaining claims are
  easy to check.
\end{proof}

\begin{lem}
	\label{lem:R1:F1}
	Let $(V,W)\in \wp _5'$ and $g,\epsilon ,z,\rho ,\sigma $
  as in Definition~\ref{def:wp} such that
  $V\simeq M(g,\rho)$ and $W\simeq M(\epsilon z,\sigma)$.
	Let $g'=g^{-1}$, $\epsilon'=\epsilon $, $z'=g^2z$, $\rho'$ be the
	character of $G^g$ dual to $\rho$, and $\sigma'$ be the
	character of $G$ given by $\sigma'(\epsilon)=1$,
	$\sigma'(g)=\rho(z)^{-1}\sigma(\epsilon)$, $\sigma'(z)=\rho(z)^2\sigma(z)$.
	Then 
  $a_{1,2}^{(V,W)}=-2$ and
	\[
		R_1(V,W)=\left(V^*,X_2^{V,W}\right)
	\]
	with $V^*\simeq M(g',\rho')$, $X_2^{V,W}\simeq M(z',\sigma')$, and 
	\begin{align*}
		\rho'(g')=\sigma'(z')=1,\quad
    1+\rho'(z')\sigma'(g')+(\rho'(z')\sigma'(g'))^2=0.
	\end{align*}
  In particular, $R_1(V,W)\in \wp _5$.
\end{lem}

\begin{proof}
	It is similar to the proof of Lemma~\ref{lem:R1:F}, but one needs
  Proposition~\ref{pro:1}.
\end{proof}

\begin{lem} \label{lem:R1:F2}
	Let $(V,W)\in \wp _5''$ and $g,\epsilon ,z,\rho ,\sigma $
  as in Definition~\ref{def:wp} such that
  $V\simeq M(g,\rho)$ and $W\simeq M(z,\sigma)$.
	Let $g'=g^{-1}$, $\epsilon'=\epsilon $, $z'=g^4z$, $\rho'$ be the
	character of $G^g$ dual to $\rho$, and $\sigma'$ be
	the  character of $G$ given by $\sigma'(g)=\rho(g)\sigma(g)$,
	$\sigma'(z)=\rho(z)^4\sigma(z)$, $\sigma'(\epsilon)=1$.  Then 
  $a_{1,2}^{(V,W)}=-4$ and
	\[
		R_1(V,W)=\left(V^*,X_4^{V,W}\right)
	\]
	with $V^*\simeq M(g',\rho')$, $X_4^{V,W}\simeq M(z',\sigma')$, and 
	\begin{align*}
	  \rho'(g'z')\sigma'(g')=1, \quad 1+\rho'(g')+\rho'(g')^2=0,\quad
          \sigma'(z')=1.
	\end{align*} 
  In particular, $R_1(V,W)\in \wp _5''$.
\end{lem}

\begin{proof}
	It is similar to the proof of Lemma~\ref{lem:R1:F}.
\end{proof}

\begin{lem}
	\label{lem:R2:F}
	Let $(V,W)\in \wp _5$ and $g,\epsilon ,z,\rho ,\sigma $
  as in Definition~\ref{def:wp} such that
  $V\simeq M(g,\rho)$ and $W\simeq M(z,\sigma)$.
	Let $g''=gz$, $\epsilon''=\epsilon^{-1}$, $z''=z^{-1}$, $\rho''$ be the
  character of $G^g$ given by $\rho''(g)=\sigma(g)$,
	$\rho''(z)=\rho(z)$, and $\sigma''$ be the
  character of $G$ dual to $\sigma$.  Then
  $a_{2,1}^{V,W}=-1$ and
	\[
	R_2(V,W)=\left( X_1^{W,V},W^* \right)
	\]
	with $X_1^{W,V}\simeq M(g'',\rho'')$, $W^*\simeq M(z'',\sigma'')$, and 
	\begin{align*}
		\rho''(g''z'')\sigma''(g'')=1, \quad
    1+\rho''(g'')+\rho''(g'')^2=0,\quad
    \sigma''(z'')=1.
	\end{align*} 
  In particular, $R_2(V,W)\in \wp _5''$.
\end{lem}

\begin{proof}
	It is similar to the proof of Lemma~\ref{lem:R1:F}.
\end{proof}

\begin{lem}
	\label{lem:R2:F1}
	Let $(V,W)\in \wp _5'$ and $g,\epsilon ,z,\rho ,\sigma $
  as in Definition~\ref{def:wp} such that
  $V\simeq M(g,\rho)$ and $W\simeq M(\epsilon z,\sigma)$.
	Let $g''=gz^2$, $\epsilon''=\epsilon^{-1}$, $z''=z^{-1}$, $\rho''$ be the
	character of $G^{g}$ given by
  $\rho''(g)=\rho(z)^{-2}\sigma(z)^{-1}$, 
  $\rho''(z)=\rho(z)\sigma(z)^2$,
	and $\sigma''$ be the
  character of $G$ dual to $\sigma$.  Then $a_{2,1}^{(V,W)}=-2$ and
	\[
	R_2(V,W)=\left( X_2^{W,V}, W^*\right)
	\]
	with $X_2^{W,V}\simeq M(g'',\rho'')$, $W^*\simeq M(z'',\sigma'')$, and 
	\begin{align*}
		\rho''(g'')=1, \quad
		\rho''(z''^2)\sigma''(\epsilon''g''^2)=1, \quad
		\sigma''(z'')=\sigma''(\epsilon''), \quad
		(3)_{\sigma''(\epsilon'')}=0.
	\end{align*} 
  In particular, $R_2(V,W)\in \wp _5'$.
\end{lem}

\begin{proof}
	It is similar to the proof of Lemma~\ref{lem:R1:F1}.
\end{proof}

\begin{lem}
	\label{lem:R2:F2}
	Let $(V,W)\in \wp _5''$ and $g,\epsilon ,z,\rho ,\sigma $
  as in Definition~\ref{def:wp} such that
  $V\simeq M(g,\rho)$ and $W\simeq M(z,\sigma)$.
  Let $g''=gz$, $\epsilon''=\epsilon ^{-1}$, $z''=z^{-1}$,
  $\rho''$ be the character of $G^g$ given by $\rho''(g)=\rho(g)\sigma(g)$,
	$\rho''(z)=\rho(z)\sigma(z)$, and $\sigma''$ be the character of
  $G$ dual to $\sigma$.  Then $a_{2,1}^{(V,W)}=-1$ and
	\[
		R_2(V,W)=\left( X_1^{W,V},W^* \right)
	\]
	with $X_1^{W,V}\simeq M(g'',\rho'')$, $W^*\simeq M(z'',\sigma'')$, and 
	\begin{align*}
		\rho''(g'')=\sigma ''(z'')=1, \quad
    (3)_{\rho ''(z'')\sigma ''(g'')}=0.
	\end{align*} 
  In particular, $R_2(V,W)\in \wp _5$.
\end{lem}

\begin{proof}
	It is similar to the proof of Lemma~\ref{lem:R1:F}.
\end{proof}

\section{Nichols algebras over $\Gamma_3$}
\label{section:NicholsG3}

\subsection{Simple Yetter-Drinfeld modules}

In this section, let $G$ be a non-abelian epimorphic image of $\Gamma _3$ and 
let $g,\epsilon ,z\in G$ such that
the group epimorphism $\Gamma _3\to G$ maps 
$\gamma \mapsto g$, $\nu \mapsto \epsilon $ and $\zeta \mapsto z$.

Let $P=\{0\}\cup \{p\in \N\,|\,p\text{ is prime}\}$.
Let $h_2=3$, $h_3=2$ and $h_p=6$ for all $p\in P\setminus \{2,3\}$,
and $h'_3=2$, $h'_p=6$ for all $p\in P\setminus \{3\}$. 
For all $1\le i\le 8$ let $\Y_i$
be the class of Yetter-Drinfeld modules $U$ over $G$ such that
there exist $x,g,\epsilon,z\in G$ and an absolutely irreducible
representation $\tau $ of $G^x$ such that $U\simeq M(x,\tau )$ and
$x$, $\tau $, and $\K $ satisfy the conditions in Table~\ref{tab:Yconditions}.
In Table~\ref{tab:Yconditions} we also provide the Hilbert series of the
Nichols algebras of the Yetter-Drinfeld modules in the classes
$\Y_i$ for all $1\le i\le 8$ and the references to these Hilbert series.

    \begin{table}
        \caption{Classes of Yetter-Drinfeld modules}
\begin{center}
        \begin{tabular}{r|l|l|c|l|l}
            & $x$ & $\tau $ & $\charK$ & Hilbert series & Ref.
	    \tabularnewline
            \hline
            $\Y_1$ & $g $ & $\tau (g )=-1$  && $(2)^2_t(3)_t$
             & \cite{MR1800714}
	    \tabularnewline
            $\Y_2$ & $g $ & $(3)_{\tau (g)}=0$ & $2$ & $(3)_t(4)_t(6)_t(6)_{t^2}$
             & \cite[Prop.\,32]{MR2891215}
	    \tabularnewline
            $\Y_3$ & $z $ & $\deg\tau =1$, $\tau (z )=-1$ && $(2)_t$
	    & \cite[\S3.4]{MR0506406}
	    \tabularnewline
            $\Y_4$ & $z $ & $\deg\tau =1$, $(3)_{-\tau (z)}=0$ & $p\in P$ &
	    $(h_p)_t$ & \cite[\S3.4]{MR0506406}
	    \tabularnewline
            $\Y_5$ & $z $ & $\deg\tau =2$, $\tau (z )=-1$ && $(2)_{t}^2$
	    & \cite[\S3.4]{MR0506406}
	    \tabularnewline
            $\Y_6$ & $\epsilon z $ & $\tau (z )=\tau (\epsilon )$,
	    $(3)_{\tau (\epsilon )}=0$ && $(3)_t^2$
	    & \cite[\S3.4]{MR0506406}
	    \tabularnewline
            $\Y_7$ & $\epsilon z $ & $\tau (\epsilon )=1$, $\tau (z )=-1$ & &
	    $(2)_t^2$ & \cite[\S3.4]{MR0506406}
	    \tabularnewline
            $\Y_8$ & $\epsilon z $ & $\tau (\epsilon z )=-1$,
             $(3)_{\tau (\epsilon )}=0$ & $p\in P$
             & $(2)_t(h'_p)_t$ & \cite[Prop.\,2.11]{MR1800709}
	     \tabularnewline
        \end{tabular}
        \label{tab:Yconditions}
\end{center}
    \end{table}

\subsection{}

In this subsection we collect the examples of Nichols algebras related to the
graph
\[
	\xymatrix{\wp_1\ar@{-}[r]^{R_1} & \wp_4}     
\]

\begin{thm}
    \label{thm:P1andP4a}
		Let $p=\charK $. Let $V,W\in \ydG $ such that $V\simeq M(g,\rho)$,
    where $\rho$ is a character of
		$G^g$, and $W\simeq M(\epsilon z,\sigma)$, where $\sigma$ is a character
    of $G^\epsilon $. Assume that $(V,W)\in \wp _1$, that is,
    \begin{align*}
      \rho(g)=-1,\quad \rho(z^2)\sigma(\epsilon g^2)=1,\quad
      1+\sigma(\epsilon)+\sigma(\epsilon)^2=0,\quad \sigma(\epsilon z)=-1.
    \end{align*}
    Then $W\in\Y_8$, $(\ad V)(W)\in\Y_1$,
    $(\ad V)^2(W)\in\Y_4$, $V\in\Y_1$, and 
    \[
    	\NA(V\oplus W)\simeq \NA(W)\otimes\NA\left( (\ad V)(W) \right) \otimes\NA\left( (\ad V)^2(W) \right)\otimes\NA(V)
    \]
    as $\N_0^2$-graded vector spaces in $\ydG$. In particular, the Hilbert series 
    of $\NA(V\oplus W)$ is
    \[
		 \mathcal{H}(t_1,t_2)=(2)_{t_2}(h'_p)_{t_2}(2)^2_{t_1t_2}
     (3)_{t_1t_2}(h_p)_{t_1^2t_2}(2)_{t_1}^2(3)_{t_1}
    \]
    and
		\[
			\dim\NA(V\oplus W)=\begin{cases}
				10368 & \text{if $\charK\not\in\{2,3\}$},\\
				5184 & \text{if $\charK=2$},\\
				1152 & \text{if $\charK=3$}.
			\end{cases}
		\]
\end{thm}

\begin{proof}
	The Cartan scheme of $(V,W)$ is standard of type $B_2$ by
  Lemma~\ref{lem:wp1wp4standard}.
  The Yetter-Drinfeld modules $(\ad V)(W)$ and $(\ad V)^2(W)$ are in the
  claimed classes because of Lemmas~\ref{lem:1:X1} and \ref{lem:1:X2}.
  Then the claims
	concerning the decomposition and the Hilbert series of $\NA(V\oplus W)$
	follow from \cite[Cor. 2.7(2)]{MR2732989}, \cite[Thm. 2.6]{MR2732989} and
	Table \ref{tab:Yconditions}.
\end{proof}

\begin{thm}
    \label{thm:P1andP4b}
		Let $p=\charK $. Let $V,W\in \ydG $ such that $V\simeq M(g,\rho)$,
    where $\rho$ is a character of
		$G^g$, and $W\simeq M(z,\sigma)$, where $\sigma$ is a character
    of $G$. Assume that $(V,W)\in \wp _4$, that is,
    \begin{align*}
        \rho(g)=-1,\quad \rho(z)\sigma(gz)=1,\quad
        1-\rho(z)\sigma(g)+(\rho(z)\sigma(g))^2=0.
    \end{align*}
		Then $W\in\Y_4$, $(\ad V)(W)\in\Y_1$,
    $(\ad V)^2(W)\in\Y_8$, $V\in\Y_1$, and 
    \[
    \NA(V\oplus W)\simeq \NA(W)\otimes\NA\left( (\ad V)(W) \right) \otimes\NA\left( (\ad V)^2(W) \right)\otimes\NA(V)
    \]
    as $\N_0^2$-graded vector spaces in $\ydG$. In particular, the Hilbert series 
    of $\NA(V\oplus W)$ is
    \[
    \mathcal{H}(t_1,t_2)=(h_p)_{t_2}(2)_{t_1t_2}^2(3)_{t_1t_2}
    (2)_{t_1^2t_2}(h'_p)_{t_1^2t_2}(2)_{t_1}^2(3)_{t_1}
    \]
    and
		\[
			\dim\NA(V\oplus W)=\begin{cases}
				10368 & \text{if $\charK\not\in\{2,3\}$},\\
				5184 & \text{if $\charK=2$},\\
				1152 & \text{if $\charK=3$}.
			\end{cases}
		\]
\end{thm}

\begin{proof}
	It is similar to the proof of Theorem \ref{thm:P1andP4a}.
  For the structure of $(\ad V)^m(W)$ for $m=1,2$ see Lemmas~\ref{lem:3:X1}
	and \ref{lem:3:X2}. Alternatively, one can deduce the claim from
	Theorem~\ref{thm:P1andP4a} since $R_1(V,W)\in \wp_1$ and $R_1^2(V,W)\simeq
(V,W)$.
\end{proof}

\subsection{}
Now we list the examples of Nichols algebras related to the graph
\[
	\xymatrix{\wp_2\ar@{-}[r]^{R_1} & \wp_3}     
\]
\begin{thm}
    \label{thm:P2andP3a}
		Let $V,W\in \ydG $ such that $V\simeq M(g,\rho)$,
    where $\rho$ is a character of
		$G^g$, and $W\simeq M(\epsilon z,\sigma)$, where $\sigma$ is a character
    of $G^\epsilon $. Assume that $\charK\ne3$ and
    $(V,W)\in \wp _2$, that is,
    \begin{align*}
        \rho(g)=\sigma(z)=-1,\quad
        \rho(z^2)\sigma(\epsilon g^2)=\sigma(\epsilon)=1.
    \end{align*}
    Then $W\in\Y_7$, $(\ad V)(W)\in\Y_1$,
    $(\ad V)^2(W)\in\Y_5$, $V\in\Y_1$, and 
    \[
    \NA(V\oplus W)\simeq \NA(W)\otimes\NA\left( (\ad V)(W) \right) \otimes\NA\left( (\ad V)^2(W) \right)\otimes\NA(V)
    \]
    as $\N_0^2$-graded vector spaces in $\ydG$. In particular, the Hilbert series 
    of $\NA(V\oplus W)$ is
    \[
    \mathcal{H}(t_1,t_2)=(2)^2_{t_2}(2)^2_{t_1t_2}(3)_{t_1t_2}(2)^2_{t_1^2t_2}(2)^2_{t_1}(3)_{t_1}
    \]
    and $\dim\NA(V\oplus W)=2^83^2=2304$. 
\end{thm}


\begin{proof}
	The Cartan scheme of $(V,W)$ is standard of type $B_2$ by
  Lemma~\ref{lem:wp2wp3standard} and  the decomposition
	of $\NA(V\oplus W)$ follows from \cite[Cor. 2.7(2) and Thm.  2.6]{MR2732989}.
	It is clear that $V\in\mathcal{Y}_1$ and $W\in\mathcal{Y}_7$. Using
  Lemmas~\ref{lem:1:X1} and \ref{lem:1:X2} we obtain that
  $(\ad V)(W)\in\mathcal{Y}_1$ and $(\ad V)^2(W)\in\mathcal{Y}_5$. Now a direct
	calculation using Table \ref{tab:Yconditions} yields the Hilbert series of
	$\NA(V\oplus W)$. 
\end{proof}

\begin{thm}
    \label{thm:P2andP3b}
		Let $V,W\in \ydG $ such that $V\simeq M(g,\rho)$,
    where $\rho$ is a character of
		$G^g$, and $W\simeq M(z,\sigma)$, where $\sigma$ is
    an absolutely irreducible representation of $G$ of degree two.
    Assume that $\charK\not=3$ and $(V,W)\in \wp _3$, that is,
    \begin{align*}
        \rho(g)=\sigma(z)=-1,\quad \rho(z^2)\sigma(g^2)=1.
    \end{align*}
    Then $W\in\Y_5$, $(\ad V)(W)\in\Y_1$,
    $(\ad V)^2(W)\in\Y_7$, $V\in\Y_1$, and 
    \[
	    \NA(V\oplus W)\simeq \NA(W)\otimes\NA\left( (\ad V)(W) \right) \otimes\NA\left( (\ad V)^2(W) \right)\otimes\NA(V)
    \]
    as $\N_0^2$-graded vector spaces in $\ydG$. In particular, the Hilbert series 
    of $\NA(V\oplus W)$ is
    \[
			\mathcal{H}(t_1,t_2)=(2)_{t_2}^2(2)_{t_1t_2}^2(3)_{t_1t_2}(2)^2_{t_1^2t_2}(2)^2_{t_1}(3)_{t_1}
    \]
    and $\dim\NA(V\oplus W)=2^83^2=2304$. 
\end{thm}


\begin{proof}
    It is similar to the proof of Theorem \ref{thm:P2andP3a}. Here one has to
    use Lemmas \ref{lem:2:X1} and \ref{lem:2:X2}.  Alternatively, one can
    deduce the claim from Theorem~\ref{thm:P2andP3a} since $R_1(V,W)\in \wp_2$
    and $R_1^2(V,W)\simeq (V,W)$.
\end{proof}

\subsection{}
Now we list the examples related to the following graph:
\[
	\xymatrix{\wp_5'\ar@{-}[r]^{R_1} & \wp_5\ar@{-}[r]^{R_2} & \wp_5''}     
\]

\begin{lem} \label{lem:Cartan3objects}
	Let $I=\{1,2\}$, $\mathcal{X}=\{a,b,c\}$, $r_1=(a\,b)\in\Sym_{\mathcal{X}}$, 
	$r_2=(b\,c)\in\Sym_{\mathcal{X}}$, 
	\begin{align*}
		A^a=\begin{pmatrix}
			2 & -2\\-2 & 2
		\end{pmatrix}, 
		&& A^b=\begin{pmatrix}
			2 & -2\\-1 & 2
		\end{pmatrix},
		&& A^c=\begin{pmatrix}
			2 & -4\\-1 & 2
		\end{pmatrix}.
	\end{align*}
	Then the following hold:
	\begin{enumerate}
		\item $\mathcal{C}=\mathcal{C}(I,\mathcal{X},(r_i)_{i\in
			I},(A^X)_{X\in\mathcal{X}})$ is a Cartan scheme.
		\item $\left(\Delta^{\re\,X}\right)_{X\in\mathcal{X}}$ forms a finite
			irreducible root system of type $\mathcal{C}$, and  
			\begin{align*}
				\Delta^{\re\,a}_+ &= \{\alpha_1,\alpha_2,\alpha_1+\alpha_2,2\alpha_1+\alpha_2,\alpha_1+2\alpha_2,2\alpha_1+3\alpha_2\},\\
				\Delta^{\re\,b}_+ &= \{\alpha_1,\alpha_2,\alpha_1+\alpha_2,2\alpha_1+\alpha_2,3\alpha_1+2\alpha_2,4\alpha_1+3\alpha_2\},\\
				\Delta^{\re\,c}_+ &= \{\alpha_1,\alpha_2,\alpha_1+\alpha_2,2\alpha_1+\alpha_2,3\alpha_1+\alpha_2,4\alpha_1+\alpha_2\}.
			\end{align*}
	\end{enumerate}
\end{lem}

\begin{proof}
	Both claims follow from the definitions. The Cartan scheme $\mathcal{C}$
	appeared already in \cite[Thm.\,6.1]{MR2498801}.
\end{proof}

\begin{thm}
	\label{thm:P5}
	Let $V,W\in \ydG $ such that $V\simeq M(g,\rho)$,
  where $\rho$ is a character of
	$G^g$, and $W\simeq M(z,\sigma)$, where $\sigma$ is
  a character of $G$.
  Assume that $\charK=2$ and $(V,W)\in \wp _5$, that is,
	\begin{align*}
		\rho(g)=\sigma (z)=1,\quad 1+\rho(z)\sigma (g)+\rho (z)^2\sigma (g)^2=0.
	\end{align*}
	Then there
	exist Yetter-Drinfeld submodules
  $W_1\in \Y _1$, $W_2\in \Y _6$, $W_3\in \Y _1$, $W_4\in \Y _3$,
  $W_5\in \Y _2$, and $W_6\in \Y _3$
	of $\NA(V\oplus W)$ of degrees
  $\alpha_1$, $2\alpha_1+\alpha_2$, $3\alpha_1+2\alpha_2$, $4\alpha_1+3\alpha_2$,
	$\alpha_1+\alpha_2$, and $\alpha_2$, respectively, such that
	\[
	\NA(V\oplus W)\simeq \NA(W_6)\otimes\NA (W_5)\otimes \cdots\otimes\NA(W_1)
	\]
	as $\N_0^2$-graded vector spaces in $\ydG$. In particular, the Hilbert series 
	of $\NA(V\oplus W)$ is
	\[
  \mathcal{H}(t_1,t_2)=(2)_{t_2}(3)_{t_1t_2}(4)_{t_1t_2}(6)_{t_1t_2}
  (6)_{t_1^2t_2^2}(2)_{t_1^4t_2^3}
  (2)_{t_1^3t_2^2}^2(3)_{t_1^3t_2^2}(3)_{t_1^2t_2}^2
  (2)_{t_1}^2(3)_{t_1}
	\]
	and $\dim\NA(V\oplus W)=2^{10}3^7=2239488$. 
\end{thm}

\begin{proof}
  Let $V',W'$ be Yetter-Drinfeld modules over $G$. If $(V',W')\in \wp_5$, then
  $V'\in \Y _1$ and $W'\in \Y _3$ by Tables~\ref{tab:pair5} and
  \ref{tab:Yconditions}.
  Similarly, if $(V',W')\in \wp_5'$, then
  $V'\in \Y _1$ and $W'\in \Y _6$, and if
  $(V',W')\in \wp_5''$, then
  $V'\in \Y _2$ and $W'\in \Y _3$.

	By assumption, $(V,W)\in \wp _5$.  Lemmas~\ref{lem:R1:F}--\ref{lem:R2:F2} and
	Lemma~\ref{lem:Cartan3objects} imply that the pair $(V,W)$ admits all
	reflections and that the set of real roots $\Delta^{\re\,(V,W)}$ is finite.
	More precisely, $|\Delta ^{\re\,(V,W)}_+|=6$.  Hence, by
	\cite[Cor.\,6.16]{MR3096611}, there exist absolutely simple
	Yetter-Drinfeld submodules $W_i\in \ydG $ of $\NA (V\oplus W)$ with $1\le
	i\le 6$, such that
  \[ \NA (V\oplus W)\simeq \NA (W_6)\otimes \NA (W_5)\otimes \cdots \otimes \NA
  (W_1).
  \]
  By the same reference, we may also assume that
  \[ \deg W_{2i+1}=(s_1s_2)^i(\alpha _1)\quad \text{and}\quad
    \deg W_{2i+2}=(s_1s_2)^is_1(\alpha _2)
  \]
  for all $0\le i\le 2$, and
  if $1\le k\le 6$ and $\deg W_k=s_{i_1}\cdots s_{i_m}(\alpha _j)$
  for some $m\ge 0$ and $i_1,\dots ,i_m,j\in \{1,2\}$,
  then $W_k$ is isomorphic in $\ydG$ to the $j$-th entry
  of $R_{i_m}\cdots R_{i_1}(V,W)$. Since $\Delta _+^{\re\,(V,W)}$
  consists of the roots
  \begin{gather*}
    s_1(\alpha _2)=2\alpha _1+\alpha _2,\quad
    s_1s_2(\alpha _1)=3\alpha _1+2\alpha _2,\quad
    s_1s_2s_1(\alpha _2)=4\alpha _1+3\alpha _2,\\
    s_1s_2s_1s_2(\alpha _1)=\alpha _1+\alpha _2, \quad
    s_1s_2s_1s_2s_1(\alpha _2)=\alpha _2,\quad \text{and}\quad \alpha _1,
  \end{gather*}
  and since
  \begin{gather*}
    R_1(V,W)\in \wp _5',\quad
    R_2R_1(V,W)\in \wp _5',\quad
    R_1R_2R_1(V,W)\in \wp _5,\\
    R_2R_1R_2R_1(V,W)\in \wp _5'',\quad
    R_1R_2R_1R_2R_1(V,W)\in \wp _5'',
  \end{gather*}
  the first paragraph of the proof implies that
  $W_1\in \Y _1$, $W_2\in \Y _6$, $W_3\in \Y _1$, $W_4\in \Y _3$,
  $W_5\in \Y _2$, $W_6\in \Y _3$.
  Finally, a direct
	calculation using Table \ref{tab:Yconditions} yields the Hilbert series of
	$\NA(V\oplus W)$. 
\end{proof}

With similar proofs,
or by applying the reflections $R_1$ and $R_2$ to pairs in $\wp
_5$,
one obtains also the following two theorems on Nichols
algebras of $V\oplus W$, where $(V,W)\in \wp_5'$ or $(V,W)\in \wp _5''$.

\begin{thm}
	\label{thm:P5'}
	Let $V,W\in \ydG $ such that $V\simeq M(g,\rho)$,
  where $\rho$ is a character of
	$G^g$, and $W\simeq M(\epsilon z,\sigma)$, where $\sigma$ is
  a character of $G^\epsilon $.
  Assume that $\charK=2$ and $(V,W)\in \wp _5'$, that is,
	\begin{align*}
		\rho(g)=1,\quad \sigma (z)=\sigma (\epsilon ),\quad
    (3)_{\sigma (\epsilon)}=0,\quad \rho (z^2)\sigma (\epsilon g^2)=1.
	\end{align*}
	Then there
	exist Yetter-Drinfeld submodules
  $W_1\in \Y _1$, $W_2\in \Y _3$, $W_3\in \Y _2$, $W_4\in \Y _3$,
  $W_5\in \Y _1$, and $W_6\in \Y _6$
	of $\NA(V\oplus W)$ of degrees
  $\alpha_1$, $2\alpha_1+\alpha_2$, $\alpha_1+\alpha_2$, $2\alpha_1+3\alpha_2$,
	$\alpha_1+2\alpha_2$, and $\alpha_2$, respectively, such that
	\[
	\NA(V\oplus W)\simeq \NA(W_6)\otimes\NA (W_5)\otimes \cdots\otimes\NA(W_1)
	\]
	as $\N_0^2$-graded vector spaces in $\ydG$. In particular, the Hilbert series 
	of $\NA(V\oplus W)$ is
	\[
  \mathcal{H}(t_1,t_2)=
  (3)_{t_2}^2 
  (2)_{t_1t_2^2}^2(3)_{t_1t_2^2}
  (2)_{t_1^2t_2^3}
  (3)_{t_1t_2}(4)_{t_1t_2}(6)_{t_1t_2}(6)_{t_1^2t_2^2}
  (2)_{t_1^2t_2}
  (2)_{t_1}^2(3)_{t_1}
	\]
	and $\dim\NA(V\oplus W)=2^{10}3^7=2239488$. 
\end{thm}

\begin{thm}
	\label{thm:P5''}
	Let $V,W\in \ydG $ such that $V\simeq M(g,\rho)$,
  where $\rho$ is a character of
	$G^g$, and $W\simeq M(z,\sigma)$, where $\sigma$ is
  a character of $G$.
  Assume that $\charK=2$ and $(V,W)\in \wp _5''$, that is,
	\begin{align*}
	  (3)_{\rho (g)}=0,\quad \sigma (z)=1, \quad \rho(gz)\sigma (g)=1.
	\end{align*}
	Then there
	exist Yetter-Drinfeld submodules
  $W_1\in \Y _2$, $W_2\in \Y _3$, $W_3\in \Y _1$, $W_4\in \Y _6$,
  $W_5\in \Y _1$, and $W_6\in \Y _3$
	of $\NA(V\oplus W)$ of degrees
  $\alpha_1$, $4\alpha_1+\alpha_2$, $3\alpha_1+\alpha_2$, $2\alpha_1+\alpha_2$,
	$\alpha_1+\alpha_2$, and $\alpha_2$, respectively, such that
	\[
	\NA(V\oplus W)\simeq \NA(W_6)\otimes\NA (W_5)\otimes \cdots\otimes\NA(W_1)
	\]
	as $\N_0^2$-graded vector spaces in $\ydG$. In particular, the Hilbert series 
	of $\NA(V\oplus W)$ is
	\[
  \mathcal{H}(t_1,t_2)=
  (2)_{t_2}
  (2)_{t_1t_2}^2(3)_{t_1t_2}
  (3)_{t_1^2t_2}^2
  (2)_{t_1^3t_2}^2(3)_{t_1^3t_2}
  (2)_{t_1^4t_2}
  (3)_{t_1}(4)_{t_1}(6)_{t_1}(6)_{t_1^2}
	\]
	and $\dim\NA(V\oplus W)=2^{10}3^7=2239488$. 
\end{thm}

\section{Proof of Theorem \ref{thm:big}}
\label{section:bigproof}

Here we prove the main result of the paper, which is Theorem~\ref{thm:big}. 

	First we prove $(1)\Rightarrow(2)$. Since $\dim\NA(V\oplus W)<\infty$, the
	pair $(V,W)$ admits all reflections by \cite[Cor.\,3.18]{MR2766176} and the
	Weyl groupoid is finite by \cite[Prop.\,3.23]{MR2766176}.

	Now we prove simultaneously $(2)\Rightarrow(3)$ and $(3)\Rightarrow(1)$. By
  \cite[Thm.\,7.3]{examples}, the quandle $\supp(V\oplus W)$ is isomorphic to 
	\[
	Z_2^{2,2},Z_3^{3,1},Z_3^{3,2},Z_4^{4,2}\text{ or }Z_T^{4,1}
	\]
	and $G$ is a non-abelian epimorphic image of its enveloping group:
	\begin{center}
		\begin{tabular}{c|c c c c c}
			Quandle &
			$Z{}_{T}^{4,1}$ &
			$Z_{2}^{2,2}$ &
			$Z_{3}^{3,1}$ &
			$Z_{3}^{3,2}$ &
			$Z_{4}^{4,2}$
			\tabularnewline
			\hline 
			Enveloping group &
			$T$ &
			$\Gamma_{2}$ &
			$\Gamma_{3}$ &
			$\Gamma_{3}$ &
			$\Gamma_{4}$
		\end{tabular}
	\end{center}
  We consider each case separately. Suppose first that $G$ is an epimorphic
  image of $\Gamma_2$ and $\supp(V\oplus W)\simeq Z_2^{2,2}$.
  Hence, by \cite[\S4]{MR2732989}, we may assume that $V\simeq M(g,\rho)$ and
  $W\simeq M(h,\sigma)$. Now the claim follows from \cite[Thm.\,4.9]{MR2732989}. 

	Now assume that $G$ is a non-abelian epimorphic image of $\Gamma_3$.
  We first prove that (2) implies (3).
  By \cite[Prop.\,4.3]{partII},
  there is a pair $(V',W')$ of absolutely simple
  Yetter-Drinfeld modules over $G$,
  which represents an object of the Weyl groupoid of $(V,W)$,
  such that the Cartan matrix of $(V',W')$ is of finite type.
  The list of examples in Subsection~\ref{subsection:G3}
  contains precisely the pairs in
  the classes $\wp _i$ for $1\le i\le 5$ and $\wp _5'$, $\wp _5''$. Hence
  this list is stable under reflections
  by Lemmas~\ref{lem:wp1wp4standard}, \ref{lem:wp2wp3standard},
  and \ref{lem:R1:F}--\ref{lem:R2:F2}.
  Thus we may assume that the Cartan matrix of $(V,W)$ is of
  finite type.
  Now (3) follows from Proposition~\ref{pro:pairs}.

  The implication (3)$\Rightarrow $(1) follows from
  Theorems~\ref{thm:P1andP4a},
  \ref{thm:P1andP4b},
  \ref{thm:P2andP3a},
  \ref{thm:P2andP3b},
  \ref{thm:P5},
  \ref{thm:P5'}, and
  \ref{thm:P5''}.

	Assume now that $G$ is a non-abelian epimorphic image of $\Gamma_4$,
  but not of $\Gamma _2$, and that
  $\supp(V\oplus W)\simeq Z_4^{4,2}$.
  Without loss of generality we may
	assume that $|\supp V|=2$ and $|\supp W|=4$.
  Let $h\in \supp V$, $g\in \supp W$, and $\epsilon =hgh^{-1}g^{-1}$.
  Then $g^G=\{g,\epsilon g,\epsilon ^2g,\epsilon ^3g\}$ and $h^G=\{h,\epsilon
  ^{-1}h\}$. Since $\supp (V\oplus W)$ generates $G$, we conclude that $G$ is
  generated by $g$ and $h$, and $V\simeq M(h,\rho)$ and $W\simeq M(g,\sigma)$
  for some character $\rho $ of $G^h$ and some character $\sigma $ of $G^g$.
  Then the implications (2)$\Rightarrow $(3) and (3)$\Rightarrow $(1)
  follow from \cite[Thm.\,5.4]{examples}.

	Finally, suppose that $G$ is an epimorphic image of $T$ and
  $\supp(V\oplus W)\simeq Z_T^{4,1}$.
  We may assume that $|\supp V|=1$ and $|\supp W|=4$.
  Let $z\in \supp V$ and $x_1\in \supp W$. Then $z$ and $x_1^G$ generate $G$.
  Choose $x_2\in x_1^G\setminus \{x_1\}$. Then $T\to G$, $\chi _1\mapsto x_1$,
  $\chi _2\mapsto x_2$, $\zeta \mapsto z$ is an epimorphism of groups, and
  $V\simeq M(z,\rho)$ and $W\simeq M(x_1,\sigma)$ for some 
  absolutely irreducible representations $\rho $ of $G$ and
  $\sigma $ of $G^{x_1}$.
  Then the implications (2)$\Rightarrow$(3) and (3)$\Rightarrow $(1) follow
  from \cite[Thm.\,2.8]{examples}.

\bigskip
\footnotesize
\noindent\textit{Acknowledgments.}
Leandro Vendramin was supported by Conicet, UBACyT 20020110300037 and the
Alexander von Humboldt Foundation.



\normalsize
\baselineskip=17pt


\def\cprime{$'$}

\end{document}